\documentclass{amsart} 

\usepackage{amsmath,amssymb,amsthm,amsfonts,amsxtra,mathtools}
\usepackage{enumerate,verbatim,mathrsfs,comment,color,url}
\usepackage{thmtools} 
\usepackage[all,2cell,ps]{xy}
\usepackage[pagebackref]{hyperref}
\usepackage{graphicx}
\usepackage{tikz-cd}
\usepackage{cleveref}
\usepackage{verbatim}

\newcommand{\up}[1]{{{}^{#1}\!}} 

\newcommand{\mycomment}[1]{}

\newcommand{\old}[1]{{\color{red} #1}}

\DeclareMathOperator{\ann}{ann}

\DeclareMathOperator{\codim}{codim}

\DeclareMathOperator{\curv}{curv}
\DeclareMathOperator{\cx}{cx}
\DeclareMathOperator{\depth}{depth}

\DeclareMathOperator{\syz}{\Omega}

\DeclareMathOperator{\Ext}{Ext}

\DeclareMathOperator{\gdim}{G-dim}

\DeclareMathOperator{\Hom}{Hom}

\DeclareMathOperator{\id}{id}

\DeclareMathOperator{\injcx}{inj\,cx}
\DeclareMathOperator{\injcurv}{inj\,curv}

\DeclareMathOperator{\pd}{pd}

\DeclareMathOperator{\tcx}{tcx}
\DeclareMathOperator{\tcurv}{tcurv}

\DeclareMathOperator{\rank}{rank}

\DeclareMathOperator{\Tor}{Tor}

\renewcommand{\ge}{\geqslant}
\renewcommand{\le}{\leqslant}

\newcommand{\fa}{\mathfrak{a}}
\newcommand{\m}{\mathfrak{m}}
\newcommand{\fn}{\mathfrak{n}}

\renewcommand{\iff}{if and only if }

\theoremstyle{plain}
\newtheorem{theorem}{Theorem}[section]
\newtheorem{lemma}[theorem]{Lemma}
\newtheorem{proposition}[theorem]{Proposition}
\newtheorem{corollary}[theorem]{Corollary}

\theoremstyle{definition}
\newtheorem{definition}[theorem]{Definition}

\newtheorem{example}[theorem]{Example}

\newtheorem{para}[theorem]{}
\newtheorem{question}[theorem]{Question}

\theoremstyle{remark}
\newtheorem{remark}[theorem]{Remark}

\numberwithin{equation}{section}

\title[Complexity and curvature of pairs of Burch modules]{Complexity and curvature of pairs of Burch modules and ideals}

\author[S.~Dey]{Souvik Dey}
\address{Department of Mathematics, University of Arkansas, Fayetteville, AR 72701, U.S.A}
\email{souvikd@uark.edu, dey0976@gmail.com}

\author[D.~Ghosh]{Dipankar Ghosh}
\address{Department of Mathematics, Indian Institute of Technology Kharagpur, West Bengal - 721302, India}
\email{dipankar@maths.iitkgp.ac.in, dipug23@gmail.com}
\urladdr{\url{https://orcid.org/0000-0002-3773-4003}}

\author[M.~Samanta]{Mouma Samanta}
\address{Department of Mathematics, Indian Institute of Technology Kharagpur, West Bengal - 721302, India}
\email{mouma17@kgpian.iitkgp.ac.in, samantastm17@gmail.com}


\begin{document}

\pagenumbering{arabic}
\thispagestyle{empty}
\begin{abstract}
    The complexity and curvature of a module were first introduced by Avramov to distinguish modules of infinite homological dimension. Later, Avramov–Buchweitz extended the notion of complexity from a single module to pairs of modules, measuring the polynomial growth rate of the minimal number of generators of their Ext-modules. By taking one of the modules in the pair to be the residue field, one recovers the standard projective and injective complexity of modules, whereas the vanishing of the complexity of a pair is equivalent to the eventual vanishing of Ext-modules, giving rise to the study of what are popularly known as Ext-pd and Ext-id test modules. Dao studied a similar notion of Tor-complexity. In the same vein, the vanishing of the Tor-complexity of pairs gives rise to Tor-pd test modules.
    On the other hand, the concept of Burch ideals was introduced by Dao–Kobayashi–Takahashi, motivated by the classical work of Burch, and subsequently extended to modules by Dey–Kobayashi.
    It follows from a result of Avramov that Burch modules exhibit extremal complexity and curvature. Moreover, Dey–Kobayashi and Ghosh–Saha showed, respectively, that Burch modules are Tor-pd and Ext-pd test, and that Burch ideals are Ext-id test. In this paper, we unify and significantly extend these two themes of extremal complexity and curvature, and Ext/Tor vanishing results of Burch modules. A key new ingredient in our proofs, particularly in dealing with Burch modules of depth zero, is the independence of the Burch property under embedding.
\end{abstract}

\maketitle

\section{Introduction}\label{Introduction}

Throughout the article, unless specified otherwise, {\it we assume $(R,\m,k)$ to be a commutative Noetherian local ring. All modules over $R$ are assumed to be finitely generated.}

The notions of complexity and curvature of a module were originally introduced by Avramov in \cite{Avr89a, Avr89b, Avr96}; see Definitions~\ref{cx-sequence} and \ref{defn:cx-injcx-curv-injcurv}. These invariants distinguish modules of infinite projective and injective dimensions. Later, in \cite{AB00}, Avramov and Buchweitz extended the notion of complexity to pairs of modules $(M,N)$, measuring the polynomial growth rate of the minimal number $\mu(\Ext^n_R(M,N))$ of generators of the Ext modules as $n$ increases. Dao introduced the analogous notion of Tor-complexity in \cite{Dao07}, which measures the polynomial growth rate of $\mu(\Tor_n^R(M,N))$. Influenced by these, in \cite{DGS24}, Dey, Ghosh and Saha developed the concept of Ext and Tor curvatures for a pair of modules, which measure the exponential growth rates of the Ext and Tor modules, respectively; see \Cref{defn:tcx-tcurv-pair-module}.

Motivated by some striking results of Burch \cite[p.~946, Thm.~5.(ii), and p.~947, Cor.~1.(ii)]{burch}, Dao, Kobayashi, and Takahashi introduced the notion of Burch ideals in \cite{DKT20}. Later, Dey and Kobayashi extended this concept by introducing Burch submodules in \cite{DK23}. A submodule $M$ of an $R$-module $L$ is called Burch if $\m(M :_L \m) \neq \m M$. A Burch submodule of $R$ is called a Burch ideal. An $R$-module $M$ is called Burch if it is isomorphic to a Burch submodule of some $R$-module. There are several well-studied classes of examples of Burch (sub)modules and ideals. We list below some interesting classes.

\begin{example}\label{exam:Burch}
\begin{enumerate}[\rm (1)]
    \item \label{ex1.1.1}
    For a submodule $M$ of an $R$-module $L$, if $\m M \neq 0$, then $\m M$ is a Burch submodule of $L$, see \cite[Ex.~3.4]{DK23}.  So a non-zero ideal $I$ of the form $\fa\m$ (for some ideal $\fa$) is a Burch ideal. In particular, if $\m^n\neq 0$ for some $n\ge 1$, then $\m^n$ is an $\m$-primary Burch ideal.
    \item 
    It follows from \eqref{ex1.1.1} and \cite[Lem.~5.2]{DK23} that Ulrich modules of dimension 1 are Burch.
    \item 
    For two $R$-modules $M$ and $N$, the direct sum $M\oplus N$ is Burch \iff at least one of $M$ or $N$ is Burch. Since $k^{\oplus \mu(\m)} \cong \m/\m^2 = \m(R/\m^2)$, in view of \eqref{ex1.1.1}, we conclude that $k$ is a Burch $R$-module. 
    \item 
    Under some mild conditions on $R$, it is shown in \cite[Prop.~4.6]{GS24} that every integrally closed ideal $I$ such that $\depth(R/I) = 0$ is Burch; in particular, $\m$-primary integrally closed ideals are Burch ideals.
    \item 
    If $R$ is of prime characteristic $p$, $F$-finite, and $\depth(R)=1$, then the $e$-th Frobenius pushforward $\up{e}{R}$ is Burch for all $e\gg 0$, see \cite[2.4.(vi)]{CDKMS}. Recall that $\up{e}{R}$ denotes the $R$-module whose underlying abelian group is $R$, and the $R$-action on $\up{e}{R}$ is given by $r \cdot x=r^{p^e}x$ for $r\in R$ and $x\in \up{e}{R}$.
\end{enumerate}
\end{example}

Burch ideals and modules possess several remarkable homological properties. For instance, they are strong test modules that detect finite homological dimensions via finite vanishing of Ext or Tor (cf.~\cite[p.~946, Thm.~5.(ii)]{burch}, \cite[Prop.~3.16]{DK23} and \cite[Thm.~5.3]{GS24}). They also characterize regular, complete intersection, Gorenstein, and CM (Cohen--Macaulay) local rings in terms of the corresponding homological dimensions (\cite[p.~947, Cor.~1.(ii)]{burch}; \cite[Cor.~3.19]{DK23}; \cite[Thms.~1.1 and 5.1.(6)]{GS24}). Furthermore, over Gorenstein rings, they satisfy the Huneke--Wiegand conjecture \cite[Thm.~1.2]{CDKMS}.

Toward understanding the growth of minimal free and injective resolutions of Burch modules, it follows from a result of Avramov \cite[Thm.~4]{Avr96} that a Burch module (and hence a Burch ideal) has maximal (projective and injective) complexity and curvature; see \Cref{prop:GS24-5.1}. As a consequence of this and \cite[Thm.~3]{Avr96}, it follows that $R$ is a complete intersection whenever the (projective or injective) complexity of some Burch $R$-module $M$ is finite, or the (projective or injective)  curvature of $M$ is at most $1$. 

In view of parts (3) and (4) of \Cref{rmk:cx-curv}, the notions of test modules, extensively studied in the literature (see, for example, \cite{Ramras, Jothilingam, HW, CDT14, DGS24b}), can be reformulated in terms of complexity as follows. Analogous descriptions also hold in terms of curvature.

\begin{definition}
    An $R$-module $M$ is said to be
    \begin{enumerate}[(1)]
        \item Tor-pd-test if for every $R$-module $L$, the condition $\tcx_R(M,L)=0$ implies that $\cx_R(L)=0$. It was introduced in \cite[Defn.~1.1]{CDT14} under the name test module.
        \item
        Ext-pd-test (resp., Ext-id-test) if for every $R$-module $L$, the condition $\cx_R(L,M)=0$ (resp., $\cx_R(M,L)=0$) implies that $\cx_R(L)=0$ (resp., $\injcx_R(L)=0$). In \cite[p.~303 and p.~301]{CDT14}, the subcategories of Ext-pd-test and Ext-id-test modules were denoted by $\operatorname{EP}(R)$ and $\operatorname{EI}(R)$ respectively. 
    \end{enumerate}
\end{definition}
It follows from \cite[Prop.~3.16]{DK23} that a Burch module is both Tor-pd-test and Ext-pd-test, while \cite[Thm.~5.3]{GS24} shows that a Burch submodule of a free module (e.g., a Burch ideal) is Ext-id-test.
        
In this paper, we aim to unify and significantly extend Avramov’s results on the extremal behavior of resolutions of Burch modules, and the aforementioned vanishing results from \cite{DK23} and \cite{GS24} for Ext and Tor modules of Burch modules that detect finite projective and injective dimensions. Indeed, we may directly  rephrase Avramov's results (i.e., \Cref{prop:GS24-5.1}) by saying that $\cx_R(k)=\tcx_R(M,k)=\cx_R(M,k)=\cx_R(k,M)$ and $\curv_R(k)=\tcurv_R(M,k)=\curv_R(M,k)=\curv_R(k,M)$ for every Burch $R$-module $M$. These discussions directly lead us to the following natural questions.

\begin{question}\label{ques-fin}
    Let $L$ and $M$ be $R$-module such that $M$ is Burch.\\
    (1) Is $\cx_R(L)\le \tcx_R(M,L)$? \; (2) Is $\cx_R(L)\le \cx_R(L,M)$? \; (3) Is $\injcx_R(L)\le \cx_R(M,L)$?
\end{question}

When $\tcx_R(M,L)=0$ or $\cx_R(L,M)=0$ or $\cx_R(M,L)=0$, the affirmative answers to \Cref{ques-fin} follow from the vanishing results \cite[Prop.~3.16]{DK23} and \cite[Thm.~5.3]{GS24} as described above, whereas the affirmative answers to \Cref{ques-fin} for $L=k$ correspond to the results of Avramov (\Cref{prop:GS24-5.1}).

A positive answer to Question~\ref{ques-fin} also implies an affirmative answer to Question~\ref{ques-char}.(1).
\begin{question}\label{ques-char}
    Let $M$ and $N$ be Burch modules.
    \begin{enumerate}[\rm (1)]
        \item If $\tcx_R(M,N)$ or $\cx_R(M,N)$ is finite, then is $R$ a complete intersection?
        \item 
        More generally, if $\tcurv_R(M,N)\le 1$ or $\curv_R(M,N)\le 1$, then is $R$ a complete intersection?
    \end{enumerate}
\end{question}


 Under certain annihilation conditions on $\Ext$ and $\Tor$, we establish various (in)equalities between the complexity (resp., curvature) of a pair of (Burch) modules and that of a single module; see Propositions~\ref{prop:tcx-tcurv-CI-ring}, \ref{prop:injcx-injcurv-CI-ring} and \ref{prop:cx-curv-CI-ring},
and Theorems~\ref{thm:tcx-tcurv-CI-ring}, \ref{thm:injcx-injcurv-CI-ring}, \ref{thm:cx-curv-CI-ring}, \ref{thm:R-CM} and \ref{new}. Combining Remark~\ref{rmk:van-cond} with some of these results, and a key insight coming from \Cref{lem:Burch-comp}, we obtain the following main results of this paper which provide partial affirmative answers to Questions~\ref{ques-fin} and \ref{ques-char}.


\begin{theorem}[See Theorems~\ref{thm:R-CM} and \ref{cornew} for stronger results]\label{thm1}
    Assume that $R$ is CM. Let $M$ and $N$ be $R$-modules such that $N$ is Burch. Suppose that either
    \begin{enumerate}[\rm (1)]
        \item 
        $M$ locally has finite projective dimension on the punctured spectrum, or
        \item 
        $\widehat R$ is locally Gorenstein on the punctured spectrum $($e.g., see {\rm \Cref{rmk:loc-gor-punc}}$)$, and $N$ locally has finite projective dimension on the punctured spectrum.
    \end{enumerate}
    Then $\cx_R(M)\le \cx_R(M,N)$ and $\curv_R(M)\le \curv_R(M,N)$.
\end{theorem}

\begin{theorem}[See Theorem~\ref{new} and \Cref{vanish} for stronger results]\label{thm2}
Let $M$ and $N$ be $R$-modules such that $M$ is a Burch submodule of some $R$-module $X$. Assume that $\gdim_R(X)<\infty$, or $X$ is torsionless. Then, the following hold.
\begin{enumerate}[\rm (1)]
    \item 
    If $M$ or $N$ locally has finite projective dimension on the punctured spectrum, then
    \[
    \cx_R(N)\le \tcx_R(M,N) \; \mbox{ and } \; \curv_R(N)\le \tcurv_R(M,N)
    \]
    \item
    If $M$ locally has finite projective dimension on the punctured spectrum, then
    \[
    \injcx_R(N)\le \cx_R(M,N)\; \mbox{ and } \;\injcurv_R(N)\le \curv_R(M,N)
    \]
    \item 
    For some $n > \max\{\depth(N),\depth(R)\}$, if $\Ext^i_R(M,N)=0$ for $i=n-1,n,n+1$, then $\id_R(N)<\infty$.
\end{enumerate}
\end{theorem}

Note that \Cref{thm2}.(3) is a variation of the injective dimension test result in \cite[Thm.~5.3]{GS24}, and it considerably strengthens \cite[Cor.~5.4]{GS24}.

Recall that $R$ is said to be a {\it complete intersection} of codimension $c$ if its completion $\widehat{R} = Q/({\bf f})$, where $(Q,\fn)$ is a regular local ring, and ${\bf f} \in \fn^2$ is a $Q$-regular sequence of length $c$. If $R$ is complete intersection of codimension at most $1$, then $R$ is called a {\it hypersurface}. In Propositions~\ref{prop:tcx-tcurv-CI-ring}, \ref{prop:injcx-injcurv-CI-ring} and \ref{prop:cx-curv-CI-ring},
and Theorems~\ref{thm:tcx-tcurv-CI-ring}, \ref{thm:injcx-injcurv-CI-ring}, \ref{thm:cx-curv-CI-ring}, \ref{thm:R-CM} and \ref{new}, we also characterize complete intersection local rings (including hypersurface rings) in terms of (Ext and Tor) complexity and curvature of pairs of Burch modules and ideals. Certain special cases of these results yield the following corollaries.

\begin{corollary}\label{cor:char-CI-via-burch-ideals}
    Let $I$ and $J$ be two Burch ideals of $R$ $($e.g., $I$ and $J$ can be any of the ideals as described in {\rm \Cref{exam:Burch}}$)$. Suppose that $I$ is $\mathfrak{m}$-primary. 
    Then $\tcx_R(I,J) = \cx_R(I,J) = \cx_R(k)$. Moreover, the following conditions are equivalent:
    \begin{enumerate}[\rm (1)]
        \item $R$ is complete intersection $($resp., of codimension $c$$)$.
        \item $\tcx_R(I,J) < \infty$ $($resp., $\tcx_R(I,J) = c < \infty$$)$.
        \item $\cx_R(I,J) < \infty$ $($resp., $\cx_R(I,J) = c < \infty$$)$.
        \item $\tcurv_R(I,J) \le 1$.
        \item $\curv_R(I,J) \le 1$.
    \end{enumerate}
    Consequently, $R$ is a hypersurface $\Longleftrightarrow$ $\tcx_R(I,J) \le 1$ $\Longleftrightarrow$ $\cx_R(I,J) \le 1$.
\end{corollary}

For $R$-modules $M$ and $N$, note that $\tcx_R(M,N) = \tcx_R(N,M)$. However, $\cx_R(M,N)$ might not be the same as $\cx_R(N,M)$.  So, a natural question arises whether the conclusions of \Cref{cor:char-CI-via-burch-ideals} hold if $J$ is $\mathfrak{m}$-primary instead of that $I$ is $\mathfrak{m}$-primary? The next result particularly ensures that the answer to this question is affirmative when $R$ is Gorenstein.


\begin{corollary}\label{cor:char-CI-via-burch-ideals-in-Gor-ring}
    Let $R$ be CM such that $\widehat R$ is locally Gorenstein on the punctured spectrum $($e.g., the rings in {\rm \Cref{rmk:loc-gor-punc}}$)$. Let $I$ and $J$ be two Burch ideals of $R$. Suppose that $J$ is $\mathfrak{m}$-primary. Then $\cx_R(I,J) = \cx_R(k)$. Moreover, the following conditions are equivalent:
    \begin{enumerate}[\rm (1)]
        \item $R$ is complete intersection $($resp., of codimension $c$$)$.
        \item $\cx_R(I,J) < \infty$ $($resp., $\cx_R(I,J) = c < \infty$$)$.
        \item $\curv_R(I,J) \le 1$.
    \end{enumerate}
    Consequently, $R$ is a hypersurface $\Longleftrightarrow$ $\cx_R(I,J) \le 1$.
\end{corollary}

The article is organized as follows. In Section~\ref{Preliminaries}, we introduce the basic notation, recall the necessary definitions, and collect the lemmas and propositions used throughout the paper. Section~\ref{Main results} contains the proofs of the main results, including Theorems~\ref{thm1} and \ref{thm2}, as well as Corollaries~\ref{cor:char-CI-via-burch-ideals} and \ref{cor:char-CI-via-burch-ideals-in-Gor-ring}. Finally, Section~\ref{sec:exam} presents several examples illustrating the main results and concludes the paper with some further questions.

\subsection*{Acknowledgments} Souvik Dey was partly supported by the Charles University Research Center program No.UNCE/24/SCI/022 and a grant GA\v{C}R 23-05148S from the Czech Science Foundation.
This work is part of the third named author's Ph.D.~thesis. She would like to thank the Government of India for financial support through the Prime Minister Research Fellowship for her Ph.D.

\section{Preliminaries}\label{Preliminaries}

In this section, we recall essential notation, definitions, and fundamental properties that are used in the rest of the article. For an $R$-module $M$, the $n$th Betti number of $M$ is defined to be $\beta_n^R(M):=\rank_k(\Ext^n_R(M,k))$, while the $n$th Bass number of $M$ is defined by $\mu_R^n(M):=\rank_k(\Ext^n_R(k,M))$. For $n\ge 0$, the $n$th syzygy of $M$ in a minimal free resolution is denoted by $\syz^n_R(M)$. We denote by $\mu_R(M)$ and $\lambda_R(M)$ the minimal number of generators and the length of $M$, respectively. Set $\ann_R(M):=\{r\in R: rx = 0\mbox{ for all }x\in M\}$. Let $\widehat{(-)}$ denote the $\m$-adic completion functor. For two $R$-modules $M$ and $N$, the notation $\Ext^{\gg 0}_R(M,N)$ (resp., $\Tor_{\gg 0}^R(M,N)$) is used to denote all $R$-modules $\Ext^{n}_R(M,N)$ (resp., $\Tor_{n}^R(M,N)$) for sufficiently large $n$.

\begin{definition}\label{cx-sequence}
    Let $\{a_n\}$ be a sequence of non-negative real numbers.
    \begin{enumerate}[\rm (1)]
        \item 
        The complexity of $\{a_n\}$, denoted by $\cx\{a_n\}$, is defined to be the smallest integer $b\ge0$ such that $a_n \le \alpha n^{b-1}$ for all $n\gg 0$ and for some constant $\alpha > 0$. If no such $b$ exists, then $\cx\{a_n\}:= \infty$.
        \item 
        The curvature of $\{a_n\}$ is defined by $\curv\{a_n\}:= \displaystyle\limsup_{n \to \infty} \sqrt[n]{a_n}$.
    \end{enumerate}
\end{definition}

\begin{lemma}{\rm \cite[Lem.~2.7 and Rmk.~2.2.(1)]{DGS24}}\label{lem:DGS-2.7-2.2}
    For sequences $\{x_n\}$ and $\{y_n\}$ of non-negative integers,  \begin{enumerate}[\rm (1)]
        \item $\cx(\{x_n + y_n\})= \sup \{\cx(\{x_n\}), \cx(\{y_n\})\}$.
        \item $\curv(\{x_n + y_n\})= \sup \{\curv(\{x_n\}), \curv(\{y_n\})\}$.
    \end{enumerate}
\end{lemma}

The complexity of a module was introduced by Avramov in \cite[Def.~3.1]{Avr89a} and \cite[Def.~1.1]{Avr89b}, while the injective complexity, referred to as plexity, was defined in \cite[Def.~5.1]{Avr89b}. The study of (injective) curvature was initiated in \cite[Sec.~1]{Avr96}.

\begin{definition}[Avramov]\label{defn:cx-injcx-curv-injcurv}
    Let $M$ be an $R$-module.
    \begin{enumerate}[\rm (1)]
        \item 
        The complexity and injective complexity of $M$ are defined to be $$\cx_R(M)=\cx\{\beta^R_n(M)\} \mbox{ and } \injcx_R(M)=\cx\{\mu_R^n(M)\}.$$
        \item 
        The curvature and injective curvature of $M$ are defined by
        $$\curv_R(M)= \curv\{\beta^R_n(M)\} \quad \mbox{and} \quad \injcurv_R(M)= \curv\{\mu_R^n(M)\}.$$
    \end{enumerate}
\end{definition}

The concepts of complexity and curvature were extended from a single module to a pair of modules by Avramov-Buchweitz, Dao and Dey-Ghosh-Saha.

\begin{definition}\label{defn:tcx-tcurv-pair-module}
    Let $M$ and $N$ be two $R$-modules.
    \begin{enumerate}[\rm (1)]
        \item 
        \cite[p.~286]{AB00} The complexity (or Ext-complexity) of $M$ and $N$ is defined to be $$\cx_R(M,N):=\cx\{\mu_R(\Ext_R^n(M,N))\}_{n\ge 0}.$$
        \item 
        \cite[Def.~3.2]{DGS24} The curvature (or Ext-curvature) of $M$ and $N$ is defined by
        $$\curv_R(M,N):=\curv\{\mu_R(\Ext_R^n(M,N))\}_{n \ge 0}.$$
        \item 
        {\rm \cite[p.~4]{Dao07}} The Tor-complexity of $M$ and $N$ is defined to be
        $$\tcx_R(M,N):=\cx\{\mu_R(\Tor^R_n(M,N))\}_{n\ge 0}.$$
        \item 
        {\rm \cite[Def.~3.3]{DGS24}} The Tor-curvature of $M$ and $N$ is defined by
        $$\tcurv_R(M,N):=\curv\{\mu_R(\Tor^R_n(M,N))\}_{n \ge 0}.$$
    \end{enumerate}
\end{definition}

\begin{remark}
    The notion of Tor-curvature was also defined in \cite[3.8]{Sega} in terms of the lengths of Tor modules. However, the definitions in terms of $\mu$ and $\lambda$ are equivalent whenever $\mathfrak{m}^h\Tor_{\gg0}^R(M,N)=0$; see \cite[Lem.~3.6.(2)]{DGS24}. In fact, in each of the definitions given in \Cref{defn:tcx-tcurv-pair-module}, $\mu$ can be replaced by $\lambda$ under suitable annihilation conditions; cf.~\cite[Lem.~2.5]{DV09} and \cite[3.4 and 3.6]{DGS24}.
\end{remark}

\begin{remark}\label{rmk:cx-curv}
    For $R$-modules $M$ and $N$, the following hold:
    \begin{enumerate}[\rm (1)]
        \item 
        $\cx_R(M) = \tcx_R(M,k) = \cx_R(M,k)$ and $\injcx_R(M) = \cx_R(k,M)$.
        \item 
        $\curv_R(M) = \tcurv_R(M,k) = \curv_R(M,k)$ and $\injcurv_R(M) = \curv_R(k,M)$.
        \item 
        \begin{enumerate}
            \item
            $\cx_R(M) = 0$ $\Longleftrightarrow$ $\curv_R(M) = 0$ $\Longleftrightarrow$ $\pd_R(M) < \infty$.
            \item
            $\injcx_R(M) = 0$ $\Longleftrightarrow$ $\injcurv_R(M) = 0$ $\Longleftrightarrow$ $\id_R(M) < \infty$.
        \end{enumerate}
        \item {\rm \cite[Rmk.~2.3 and 3.8.(3)]{DGS24}}
        \begin{enumerate}
            \item $\cx_R(M,N) < \infty$ $\Longrightarrow$ $\curv_R(M,N) \le 1$ $\Longleftrightarrow$ $\curv_R(M,N) = 0$ or $1$.
            \item $\tcx_R(M,N) < \infty$ $\Longrightarrow$ $\tcurv_R(M,N) \le 1$ $\Longleftrightarrow$ $\tcurv_R(M,N) = 0$ or $1$.
            \item 
            $\cx_R(M,N) = 0$ $\Longleftrightarrow$ $\curv_R(M,N) = 0$ $\Longleftrightarrow$ $\Ext_R^{\gg 0}(M,N) = 0$.
            \item 
            $\tcx_R(M,N) = 0$ $\Longleftrightarrow$ $\tcurv_R(M,N) = 0$ $\Longleftrightarrow$ $\Tor^R_{\gg 0}(M,N) = 0$.
        \end{enumerate}
    \end{enumerate}
\end{remark}

Among all the $R$-modules, the residue field of $R$ has maximal (injective) complexity and curvature. These also characterize complete intersection local rings.

\begin{proposition}\label{prop:properties-cx-curv}
    \begin{enumerate}[\rm (1)]
        \item 
        {\rm \cite[Prop.~2]{Avr96}} For any $R$-module $M$, the following hold.
    \begin{enumerate}[\rm (i)]
        \item $\cx_R(M)\le \cx_R(k)$ and $\injcx_R(M) \le \injcx_R(k) = \cx_R(k)$.
        \item $\curv_R(M) \le \curv_R(k)$ and $\injcurv_R(M) \le \injcurv_R(k) =\curv_R(k)$.
    \end{enumerate}
        \item {\rm \cite[Thm.~3]{Avr96}} The following conditions are equivalent:
    \begin{enumerate}[\rm (i)]
        \item $R$ is complete intersection $($resp., of codimension $c$$)$.
        \item $\cx_R(k)<\infty$ $($resp., $\cx_R(k) = c < \infty$$)$.
        \item $\curv_R(k) \le 1$.
    \end{enumerate}
  \end{enumerate}
\end{proposition}

The (injective) complexity and curvature are preserved under taking quotient by a regular element of both the ring and the module.

\begin{lemma}\label{lem:cx-curv-mod-x}
Let $M$ be an $R$-module. Let $x$ be an element of $R$ which is regular on both $R$ and $M$. Then 
    \begin{enumerate}[\rm (1)]
        \item 
        $\cx_R(M) = \cx_{R/xR}(M/xM)$ and $\curv_R(M) = \curv_{R/xR}(M/xM)$. 
        \item 
        $\injcx_R(M) = \injcx_{R/xR}(M/xM)$ and $\injcurv_R(M) = \injcurv_{R/xR}(M/xM)$. 
    \end{enumerate}
\end{lemma}

\begin{proof}
    In view of \cite[p.~140, Lem.~2]{Mat89}, $\Ext_R^{n}(M,k) \cong \Ext_{R/xR}^n(M/xM,k)$ and $\Ext_R^{n+1}(k,M) \cong \Ext_{R/xR}^n(k,M/xM)$ for all $n\ge 0$. So $\beta_n^R(M)=\beta_n^{R/xR}(M/xM)$ and $\mu^{n+1}_R(M)=\mu^n_{R/xR}(M/xM)$ for all $n \ge 0$. Hence the desired equalities follow.
    %
\end{proof}


Under certain conditions, (Ext and Tor) complexity of a pair of modules can be bounded above by complexity of each of the modules.

\begin{proposition}\label{prop:AB00-II-DGS24-7.9}
    Let $R$ be complete intersection, and $M$, $N$ be two $R$-modules. Then the following hold.
    \begin{enumerate}[\rm (1)]
        \item 
        $\cx_R(M)=\injcx_R(M)$, and $\cx_R(M, N) \le \min\{\cx_R(M), \cx_R(N)\} < \infty$ and $\curv_R(M,N) \le 1$.
        \item 
        Suppose that $\m^h\Tor_{\gg 0}^R(M,N)=0$ for some $h \ge 0$. Then $\tcx_R(M, N) \le \min\{\cx_R(M), \cx_R(N)\} < \infty$ and $\tcurv_R(M,N) \le 1$.
    \end{enumerate}
\end{proposition}

\begin{proof}
    (1) It is shown in \cite[Thm.~II]{AB00} that $\cx_R(M)=\injcx_R(M)$ and $\cx_R(M, N) \le \min\{\cx_R(M), \cx_R(N)\}$. By \Cref{prop:properties-cx-curv}.(1), $\min\{\cx_R(M), \cx_R(N)\}\le \cx_R(k) < \infty$. Thus $\cx_R(M, N) < \infty$. Hence, in view of \Cref{rmk:cx-curv}.(4), $\curv_R(M,N) \le 1$.

    (2) The inequalities for (Tor) complexities are shown in \cite[Cor.~5.7]{Dao07} and \cite[Prop.~7.9.(1)]{DGS24}. Hence, by \Cref{rmk:cx-curv}.(4), $\tcurv_R(M,N) \le 1$.
\end{proof}

Both complexity and curvature behave well with finite direct sum of modules.

\begin{lemma}\label{lem:cx-curv-direct-sum}
    Let $M$, $N$, $M_i$ and $N_i$ be $R$-modules for $i=1,2$. Then 
    \begin{enumerate}[\rm (1)]
            \item 
            $\tcx_R(M_1 \oplus M_2, N) = \sup\{\tcx_R(M_1,N), \tcx_R(M_2,N)\}$.  
            \item 
            $\tcurv_R(M_1 \oplus M_2, N) = \sup\{\tcurv_R(M_1,N), \tcurv_R(M_2,N)\}$.
            \item 
            $\cx_R(M_1 \oplus M_2, N) = \sup\{\cx_R(M_1,N), \cx_R(M_2,N)\}$.
            \item 
            $\curv_R(M_1 \oplus M_2, N) = \sup\{\curv_R(M_1,N), \curv_R(M_2,N)\}$.
            \item 
            $\cx_R(M,N_1 \oplus N_2) = \sup\{\cx_R(M,N_1), \cx_R(M,N_2)\}$.
            \item 
            $\curv_R(M,N_1 \oplus N_2) = \sup\{\curv_R(M,N_1), \curv_R(M,N_2)\}$.
        \end{enumerate}
\end{lemma}

\begin{proof}
    Note that for any two $R$-modules $L_1$ and $L_2$, one has that $\mu_R(L_1\oplus L_2)= \mu_R(L_1)+\mu_R(L_2)$. Hence, since $\Tor^R_n(M_1 \oplus M_2, N)\cong \Tor^R_n(M_1,N) \oplus \Tor^R_n(M_2,N)$ for all $n$, the proofs of (1) and (2) are direct consequences of \Cref{lem:DGS-2.7-2.2}. Similarly, considering Ext modules, one obtains the other equalities.
\end{proof}

\begin{para}\label{para:fact-on-exact-seq}
    Let $X\to Y\to Z$ be an exact sequence of $R$-modules. Then, $(\ann_RX)(\ann_RZ)\subseteq \ann_R Y$. Moreover, if all $X$, $Y$ and $Z$ have finite lengths, then $\lambda_R(Y)\le \lambda_R(X)+\lambda_R(Z)$.
\end{para}

\begin{lemma}\label{lem:cx-curv-s.e.s}
Let $L,M_1,M_2,M_3$ and $N$ be $R$-modules such that there is a short exact sequence $0\to M_1\to M_2 \to M_3 \to 0$. Let $s,t,u \in \{1,2,3\}$ be distinct $($e.g., $s=2$, $t=1$ and $u=3$$)$.
\begin{enumerate}[\rm (1)]
    \item
    Suppose that $\m^h\Tor^R_{\gg 0}(M_s,N)=0$ and $\m^h\Tor^R_{\gg 0}(M_t,N)=0$ for some $h>0$. Then, $\m^{2h}\Tor^R_{\gg 0}(M_u,N)=0$. Moreover, $\tcx_R(M_u,N)\le \sup\{\tcx_R(M_s,N), \tcx_R(M_t,N)\}$ and $\tcurv_R(M_u,N)\le \sup\{\tcurv_R(M_s,N), \tcurv_R(M_t,N)\}$.
    \item
    Suppose that $\m^h\Ext^{\gg 0}_R(M_s,N)=0$ and $\m^h\Ext^{\gg 0}_R(M_t,N)=0$ for some $h>0$. Then, $\m^{2h}\Ext^{\gg 0}_R(M_u,N)=0$. Moreover, $\cx_R(M_u,N)\le \sup\{\cx_R(M_s,N), \cx_R(M_t,N)\}$ and $\curv_R(M_u,N)\le \sup\{\curv_R(M_s,N), \curv_R(M_t,N)\}$.
    \item
    Suppose that $\m^h\Ext^{\gg 0}_R(L,M_s)=0$ and $\m^h\Ext^{\gg 0}_R(L,M_t)=0$ for some $h>0$. Then, $\m^{2h}\Ext^{\gg 0}_R(L,M_u)=0$. Moreover, $\cx_R(L,M_u)\le \sup\{\cx_R(L,M_s), \cx_R(L,M_t)\}$ and $\curv_R(L,M_u)\le \sup\{\curv_R(L,M_s), \curv_R(L,M_t)\}$. 
\end{enumerate}
\end{lemma}

\begin{proof}
    (1) Consider the corresponding long exact sequence of Tor modules. Hence, the first claim follows from \ref{para:fact-on-exact-seq}, while the desired inequalities can be obtained from \ref{para:fact-on-exact-seq}, \cite[Lem.~3.4]{DGS24} and \Cref{lem:DGS-2.7-2.2} by computing the lengths of the Tor modules.

    Similarly, considering the long exact sequences of Ext modules, we obtain (2) and (3).
\end{proof}

\begin{lemma}\label{newloc}
    An $R$-module $M$ locally has finite projective dimension on punctured spectrum of $R$ if and only if the completion $\widehat M$ locally has finite projective dimension on punctured spectrum of $\widehat R$.
\end{lemma}

\begin{proof}
    Set $d:=\dim(R)$. Then, $M$ locally has finite projective dimension on punctured spectrum of $R$ if and only if the $d$th syzygy $\syz^d_R(M)$ is locally free on punctured spectrum of $R$, if and only if $\Ext^1_R\big(\syz^d_R(M), \syz^{d+1}_R(M)\big)$ has finite length (cf.~\cite[Prop.~2.10]{nfloc}), if and only if $ \Ext^1_{\widehat R}\big(\syz^d_{\widehat R}(\widehat M), \syz^{d+1}_{\widehat R}( \widehat M)\big)$ has finite length, if and only if  $\syz^d_{\widehat R}(\widehat M)$ is locally free on punctured spectrum of $\widehat R$, if and only if $\widehat M$ locally has finite projective dimension on punctured spectrum of $\widehat R$. 
\end{proof}

Next, we recall the notions of Burch (sub)modules and ideals.

\begin{definition}\label{defn:Burch}
\begin{enumerate}[(1)]
    \item \cite[Def.~3.1]{DK23} 
    A submodule $M$ of an $R$-module $L$ is said to be Burch if
    \begin{center}
        $\m(M :_L \m) \neq \m M$, i.e., $\m(M :_L \m) \nsubseteq \m M$.
    \end{center}
    An $R$-module $M$ is called Burch if it is a Burch submodule of some $R$-module.
    \item \cite[Def.~2.1]{DKT20} 
    A Burch submodule of $R$ is called a Burch ideal of $R$. 
\end{enumerate}
\end{definition}

The following property of Burch modules follows from (1)$\Leftrightarrow$(2) or (1)$\Leftrightarrow$(4) in \cite[Lem.~3.9]{DK23}.

\begin{lemma}\label{burchcomplete}
    If $M$ is a Burch $R$-module, then $\widehat M$ is a Burch $\widehat R$-module. 
\end{lemma}

A key property of a Burch module of positive depth is the following. 

\begin{lemma}\label{lem:DK23-3.6}{\rm \cite[Lem.~3.6]{DK23}}
     Let $M$ be a Burch submodule of an $R$-module $L$ such that $\depth(M) > 0$. Then there exists an $M$-regular element $x \in \m$ such that $k$ is a direct summand of $M/xM$. 
\end{lemma}

Inherent to the definition of Burch submodule is the role of how a module $M$ embeds in another module $L$. We may wonder how important the role of this bigger module is.

\begin{question}\label{ques}
    Let $f:M\to X$ and $g:M\to Y$ be injective $R$-module homomorphisms. If $f(M)$ is a Burch submodule of $X$, then is $g(M)$ a Burch submodule of $Y$?
\end{question}

In order to give a partial answer to \Cref{ques}, we prove the following.  

\begin{lemma}\label{le}
    Let $U, V$ and $W$ be $R$-modules such that $U\subseteq V$ and there exists an injective homomorphism $f:V\to W$. Consider an ideal $I$ of $R$. Then, $f\big((U :_V I)\big) = \big(f(U):_{f(V)}I\big)$. 
\end{lemma}

\begin{proof}
    First, note that $I f\big((U:_V I)\big) = f\big(I(U:_V I)\big) \subseteq f(U)$. Thus, $f\big((U:_V I)\big)\subseteq \big(f(U):_{f(V)}I\big)$. For reverse inclusion, let $w\in \big(f(U):_{f(V)}I\big)$. Then, $w=f(v)$ for some $v\in V$, and $f(Iv)=If(v)=Iw\subseteq f(U)$. Hence, since $f$ is injective, $Iv\subseteq U$, i.e., $v\in (U:_V I)$. So $w=f(v)\in f\big((U:_VI)\big)$.
\end{proof}

In a very special case, the next lemma provides an affirmative answer to \Cref{ques}, and this will be used extensively in the next section.

\begin{lemma}\label{lem:Burch-comp}
    Let $f:M\to X$ and $h:X\to Y$ be injective $R$-module homomorphisms. If $f(M)$ is a Burch submodule of $X$, then $h(f(M))$ is a Burch submodule of $Y$. 
\end{lemma}

\begin{proof}
    Put $U:=f(M)$. Then $U$ is a submodule of $X$, and $h(U)$ is a submodule of $h(X)$. If possible, suppose that $h(U)$ is not a Burch submodule of $Y$. Then, $\m h(U)=\m\big(h(U):_Y \m\big)$. Hence the inclusions
    \[
        \m h(U)\subseteq \m \big(h(U):_{h(X)}\m\big) \subseteq \m\big(h(U):_Y \m\big)
    \]
    imply that $\m h(U)=\m \big(h(U):_{h(X)}\m\big) =\m \,h\big((U:_X \m)\big)$, where the last equality is obtained from \Cref{le}. Thus, $h(\m U)=h\big(\m(U:_X \m)\big)$. Since $h$ is injective, it follows that $\m U=\m(U:_X \m)$, i.e., $U=f(M)$ is not Burch in $X$, which is a contradiction. So $h(f(M))$ is a Burch submodule of $Y$.
\end{proof}

Burch modules have maximal (injective) complexity and curvature.

\begin{proposition}\label{prop:GS24-5.1}{\rm \cite[Thm.~4]{Avr96}, \cite[Thm.~5.1]{GS24}}
    Let $M$ be a Burch module. Then
    \begin{enumerate}[\rm (1)]
        \item $\cx_R(M)=\cx_R(k)$ and $\injcx_R(M)=\cx_R(k)$.
        \item $\curv_R(M)=\curv_R(k)$ and $\injcurv_R(M)=\curv_R(k)$.
    \end{enumerate}
\end{proposition}

The next lemma is important for proving results about Burch modules of depth zero.

\begin{lemma}\label{lem:GS24-lem-4.8}{\rm \cite[Lem.~4.8]{GS24}}
    Let $S= R[X]_{\langle \m, X \rangle}$. Let $M$ be a submodule of an $R$-module $L$. Set $L':= L[X]_{\langle \m,X\rangle}$ and $M' := (M + XL[X])_{\langle \m,X\rangle}$. Then
    \begin{enumerate}[\rm (1)]
        \item If $M$ is a Burch submodule of $L$, then $M'$ is a Burch $S$-submodule of $L'$.
        \item $M'/XM'\cong M \oplus L/M$ as $R$-modules.
    \end{enumerate}
\end{lemma}


\begin{remark}
    In the common hypothesis of all the statements in \cite[Lem.~4.8]{GS24}, it is assumed that $M$ is a Burch submodule of $L$. Note that the Burch property of $M$ is required for the proof of statement (2) in \cite[Lem.~4.8]{GS24}, and is not needed elsewhere.
\end{remark}

The following lemma ensures that the (injective) complexity and curvature of a module are preserved under the base change from $R$ to $R[X]_{\langle \m,X\rangle}$.

\begin{lemma}\label{lem:cx-curv-N}
    Let $N$ be an $R$-module. Set $S := R[X]_{\langle \m,X\rangle}$. Then $S$ is a local ring with the same residue field as that of $R$. We may consider $N$ as an $S$-module via the isomorphism $S/XS\cong R$. Moreover, the following hold.
    \begin{enumerate}[\rm (1)]
        \item $\cx_S(N) = \cx_R(N)$ and $\curv_S(N) = \curv_R(N)$.
        \item $\injcx_S(N) = \injcx_R(N)$ and $\injcurv_S(N) = \injcurv_R(N)$.
    \end{enumerate}
\end{lemma}

\begin{proof}
    Set $N' := N[X]_{\langle \m,X\rangle}$. Note that $X$ is $N'$-regular, and $N'/XN'\cong N$.
    
    (1) The first equality is obtained as follows:
    \begin{align*}
        \cx_S(N)= \cx_S(N'/XN') & = \cx_S(N') \mbox{ \ \ (by \cite[Lem.~3.11.(1)]{DGS24})}\\
        & = \cx_{S/XS}(N'/XN') \mbox{ \ \ (by \Cref{lem:cx-curv-mod-x}.(1))}\\
        & = \cx_R(N).
    \end{align*}
    The same equalities hold for curvature as well.
    
    (2) Similarly, as in (1), the desired equalities on injective complexity and curvature can be obtained by using \cite[Lem.~3.11.(2)]{DGS24} and \Cref{lem:cx-curv-mod-x}.(2).
\end{proof}

For a submodule $M$ of $L$, we now see how Ext and Tor modules involving $M$ behave if $M$ over $R$ is changed to $M' := (M + XL[X])_{\langle \m,X\rangle}$ over $R[X]_{\langle \m,X\rangle}$.

\begin{lemma}\label{lem:Tor-Ext-iso-over-S-R}
    Let $M$ be a submodule of an $R$-module $L$, and let $N$ be an $R$-module. Set $S := R[X]_{\langle \m,X\rangle}$, $L':= L[X]_{\langle \m,X\rangle}$ and $M' := (M + XL[X])_{\langle \m,X\rangle}$. We may consider $N$ as an $S$-module via the isomorphism $S/XS\cong R$. 
    Then we have the following $R$-module isomorphisms.
    \begin{enumerate}[\rm (1)]
        \item $\Tor_n^S(M',N) \cong \Tor_n^R(M, N) \oplus \Tor_n^R(L/M, N)$ for every $n\ge 0$.
        \item $\Ext^n_S(M',N) \cong \Ext^n_R(M, N) \oplus \Ext^n_R(L/M, N)$ for every $n\ge 0$.
        \item $\Ext^{n+1}_S(N,M') \cong \Ext^{n}_R(N, M) \oplus \Ext^{n}_R(N,L/M)$ for every $n\ge 0$.
    \end{enumerate}
\end{lemma}

\begin{proof}
     Note that $X$ is a regular element on $S$ and $M'$. Therefore
    \begin{align}
        \Tor_n^S(M',N) & \cong \Tor_n^{S/XS}(M'/XM',N) \mbox{ \quad (by \cite[p.~140, Lem.~2]{Mat89})} \label{Mat.p140.lem2}\\
        & \cong \Tor_n^R(M\oplus L/M, N) \mbox{ \quad  (by \Cref{lem:GS24-lem-4.8}.(2))} \nonumber \\
        & \cong \Tor_n^R(M,N) \oplus \Tor_n^R(L/M,N) \mbox{ for all $n\ge0$}.\nonumber
    \end{align}
    Replacing Tor by Ext, similar isomorphisms as in \eqref{Mat.p140.lem2} yield (2). 
    Since
    \begin{align*}
        \Ext^{n+1}_S(N,M') & \cong \Ext^{n}_{S/XS}(N,M'/XM') \mbox{\quad (by \cite[p.~140, Lem.~2]{Mat89})}\\
        & \cong \Ext^n_R(N,M\oplus L/M) \mbox{\quad (by \Cref{lem:GS24-lem-4.8}.(2))}\\
        & \cong \Ext^n_R(N,M) \oplus \Ext^n_R(N,L/M) \mbox{ for all $n\ge0$},
    \end{align*}
    the proof of (3) follows.
\end{proof}


Some consequences of \Cref{lem:Tor-Ext-iso-over-S-R} are the following.

\begin{lemma}\label{lem:cx-curv-over-S-R}
    With the hypotheses as in {\rm \Cref{lem:Tor-Ext-iso-over-S-R}}, the following hold.
    \begin{enumerate}[\rm (1)]
        \item $\tcx_S(M',N) = \sup\{ \tcx_R(M,N), \tcx_R(L/M,N) \}$.
        \item $\tcurv_S(M',N) = \sup\{ \tcurv_R(M,N), \tcurv_R(L/M,N) \}$.
        \item $\cx_S(M',N) = \sup\{\cx_R(M,N), \cx_R(L/M,N)\}$.
        \item $\curv_S(M',N) = \sup\{\curv_R(M,N), \curv_R(L/M,N)\}$.
        \item $\cx_S(N,M') = \sup\{\cx_R(N,M), \cx_R(N,L/M)\}$.
        \item $\curv_S(N,M') = \sup\{\curv_R(N,M), \curv_R(N,L/M)\}$.
    \end{enumerate}
\end{lemma}

\begin{proof}
    Since for any two $R$-modules $U$ and $V$, one has $\mu_R(U\oplus V)= \mu_R(U)+\mu_R(V)$, all the equalities are direct consequences of Lemmas~\ref{lem:DGS-2.7-2.2} and \ref{lem:Tor-Ext-iso-over-S-R} simultaneously.
\end{proof}

The notion of Gorenstein dimension (in short, G-dimension) is due to Auslander, and it is developed by Auslander and Bridger in \cite{AB69}.

\begin{definition}
\begin{enumerate}[\rm (1)]
    \item An $R$-module $G$ is said to have Gorenstein dimension zero, i.e., $\gdim_R(G)=0$, if $G$ is reflexive, and $\Ext^i_R(G, R) = \Ext^i_R(G^*, R) = 0$ for all $i \ge 1$, where $G^* := \Hom_R(G,R)$.
    \item For an $R$-module $M$, $\gdim_R(M)$ is defined to be the infimum of all integers $n \ge 0$ such that there exists an exact sequence $0\rightarrow G_{n} \rightarrow G_{n-1}\rightarrow \cdots \rightarrow G_0\rightarrow M\rightarrow 0$, where $\gdim_R(G_i) = 0$ for all $0 \le i \le n$. If no such $n$ exists, then $\gdim_R(M):=\infty$.
\end{enumerate}
\end{definition}
%

\begin{para}\label{para:Aus-Buch-Thm-1.1}
    Let $\mathcal{A}$ be the category of $R$-modules of Gorenstein dimension zero, and $\mathcal{B}$ be the subcategory of projective $R$-modules. Following the notation as in \cite[p.~9]{Aus-buch}, $\hat{\mathcal{A}}$ and $\hat{\mathcal{B}}$ denote the category of $R$-modules of finite Gorenstein dimension and of finite projective dimension, respectively. Then, in view of \cite[Thm.~1.1]{Aus-buch}, for each $U$ in $\hat{\mathcal{A}}$, there is an exact sequence $0 \rightarrow U \rightarrow V \rightarrow W \rightarrow 0$ with $V$ in $\hat{\mathcal{B}}$ and $W$ in $\mathcal{A}$.
\end{para}

\section{Main results}\label{Main results}

Here, we provide the main results of this article. We start with the following proposition.

\begin{proposition}\label{prop:tcx-tcurv-CI-ring}
    Let $M$ and $N$ be $R$-modules such that $M$ is Burch. Let $\depth(M)\ge 1$, and $ \mathfrak{m}^h \Tor^R_{\gg 0}(M,N) = 0 $ for some $h\ge 1$. Then
    \begin{enumerate}[\rm (1)]
        \item 
        $\cx_R(N) \le \tcx_R(M,N)$ and $\curv_R(N) \le \tcurv_R(M,N)$. When $R$ is a complete intersection, both inequalities become equalities.
        \item 
        Suppose $N$ is also a Burch module. Then the following are equivalent. 
        \begin{enumerate}[\rm (i)]
           \item $R$ is complete intersection $($resp., of codimension $c$$)$.
           \item $\tcx_R(M,N) < \infty$ $($resp., $\tcx_R(M,N) = c < \infty$$)$.
           \item $\tcurv_R(M,N) \le 1$.
        \end{enumerate}
    \end{enumerate}
\end{proposition}

\begin{proof}
    (1) By \Cref{lem:DK23-3.6}, there exists an $M$-regular element $x\in\m$ such that $k$ is a direct summand of $M/xM$. Therefore  $\cx_R(N)=\tcx_R(k,N)\le \tcx_R(M/xM,N)\le \tcx_R(M,N)$, where the last two inequalities are obtained from \Cref{lem:cx-curv-direct-sum}.(1) and \cite[Lem.~3.12.(2)]{DGS24} respectively. Similarly, one obtains $\curv_R(N) \le \tcurv_R(M,N)$. 
    
    Suppose that $R$ is complete intersection. Then, by \Cref{prop:AB00-II-DGS24-7.9}.(2), one has $\tcx_R(M,N)\le \min\{\cx_R(M), \cx_R(N)\} \le \cx_R(N)$. Hence $\cx_R(N)=\tcx_R(M,N)$. For equality on the curvatures, since $0\le \curv_R(N) \le \tcurv_R(M,N)$, there is nothing to prove when $\tcurv_R(M,N)=0$. So we may assume that $\tcurv_R(M,N) > 0$, equivalently, $\tcx_R(M,N)>0$, cf.~\Cref{rmk:cx-curv}.(4).(d). Note that $\tcx_R(M,N) = \cx_R(N) < \infty$ (by \Cref{prop:properties-cx-curv}). Hence, by \Cref{rmk:cx-curv}.(4), $0<\tcurv_R(M,N) \le 1$, i.e., $\tcurv_R(M,N)=1$. It remains to prove that $\curv_R(N) = 1$. Since $R$ is complete intersection, $\curv_R(N) \le 1$, i.e., $\curv_R(N)$ can be $0$ or $1$. Since $\cx_R(N)=\tcx_R(M,N)>0$, it follows that $\curv_R(N)=1$. Thus, $\curv_R(N) = \tcurv_R(M,N)$.

    (2) First we note that if $R$ is complete intersection, in view of \Cref{prop:properties-cx-curv}.(2),
    \begin{equation}\label{cdim.equal.tcx}
        \codim(R) = \cx_R(k) = \cx_R(N) = \tcx_R(M,N), 
    \end{equation}
    where the last two equalities are obtained from \Cref{prop:GS24-5.1}.(1) and (1) respectively. So, the implications (i) $\Rightarrow$ (ii) and (i) $\Rightarrow$ (iii) follow from \Cref{prop:AB00-II-DGS24-7.9}.(2),
    while the implication (ii) $\Rightarrow$ (iii) is trivial (\Cref{rmk:cx-curv}.(4).(b)). Assuming (iii), by \Cref{prop:GS24-5.1}.(2) and (1), $\curv_R(k)=\curv_R(N) \le \tcurv_R(M,N) \le 1$, and hence $R$ is complete intersection by \Cref{prop:properties-cx-curv}.(2). Thus, one also sees that if $\tcx_R(M,N) = c < \infty$, then $R$ is complete intersection of codimension $c$, see \eqref{cdim.equal.tcx}.
\end{proof}

For Burch modules of arbitrary depth, we prove the following result on Tor complexity and curvature.

\begin{theorem}\label{thm:tcx-tcurv-CI-ring}
    Let $L,M$ and $N$ be $R$-modules such that $M$ is a Burch submodule of $L$. Let $\m^h\Tor^R_{\gg 0}(M,N)=\m^h\Tor^R_{\gg 0}(L,N)=0$ for some $h\ge 1$. Then 
    \begin{enumerate}[\rm (1)]
        \item 
        $\cx_R(N) \le \max\{\tcx_R(M,N), \tcx_R(L/M,N) \}$.
        \item 
        $\curv_R(N) \le \max\{\tcurv_R(M,N), \tcurv_R(L/M,N) \}$.
        \item 
        Both inequalities {\rm (1)} and {\rm (2)} become equalities if $R$ is a complete intersection.
        \item 
        Suppose $N$ is also a Burch module. Then the following are equivalent.
        \begin{enumerate}[\rm (i)]
        \item $R$ is complete intersection. 
        \item $\max\{\tcx_R(M,N), \tcx_R(L/M,N) \} < \infty$. 
        \item $\max\{\tcurv_R(M,N), \tcurv_R(L/M,N) \} \le 1$.
        \end{enumerate}
        Under these conditions, $\codim(R) = \max\{\tcx_R(M,N), \tcx_R(L/M,N) \}$.
    \end{enumerate}
\end{theorem}

\begin{proof}
    (1) When $\depth(M)\ge 1$, the inequality follows from \Cref{prop:tcx-tcurv-CI-ring}.(1). So, we may assume that $\depth(M)=0$. Consider $S := R[X]_{\langle \m,X\rangle}$, $L':= L[X]_{\langle \m,X\rangle}$ and $M' := (M + XL[X])_{\langle \m,X\rangle}$. We may view $N$ as an $S$-module via the isomorphism $S/XS\cong R$. Since $\m^h\Tor^R_{\gg 0}(M,N)=0$ and $\m^h\Tor^R_{\gg 0}(L,N)=0$, it follows from \Cref{lem:cx-curv-s.e.s}.(1) that $\m^{2h}\Tor^R_{\gg 0}(L/M,N)=0$. Hence, in view of \Cref{lem:Tor-Ext-iso-over-S-R}.(1), $\fn^{2h}\Tor^S_{\gg 0}(M',N)=0$, where $\fn$ is the maximal ideal of $S$. Note that $M'$ is a Burch $S$-submodule of $L'$ by \Cref{lem:GS24-lem-4.8}.(1). Therefore, since $\depth_S(M')\ge 1$, by \Cref{prop:tcx-tcurv-CI-ring}.(1), $\cx_S(N) \le \tcx_S(M',N)$. Due to \Cref{lem:cx-curv-N}.(1), one has $\cx_S(N) = \cx_R(N)$. Hence, by \Cref{lem:cx-curv-over-S-R}.(1), $\cx_R(N) \le \tcx_S(M',N) = \max\{ \tcx_R(M,N), \tcx_R(L/M,N) \}$. 
        
   (2) Replacing $\cx$ by $\curv$, the proof follows along the same lines as (1). 

   (3) Assume that $R$ is a complete intersection. Then, by \Cref{prop:AB00-II-DGS24-7.9}.(2), $\tcx_R(M,N) \le \min\{ \cx_R(M), \cx_R(N) \} \le \cx_R(N)$. Similarly, $\tcx_R(L/M,N) \le \cx_R(N)$. Thus,
   \[
    \max\{ \tcx_R(M,N), \tcx_R(L/M,N) \}\le \cx_R(N).
   \]
   Hence, by (1), the equality follows. Analogous argument shows the equality for curvature.

   (4) In view of \Cref{prop:GS24-5.1}, $\cx_R(N)=\cx_R(k)$ and $\curv_R(N) = \curv_R(k)$. So the desired equivalences follow from (1), (2), \Cref{prop:properties-cx-curv}.(2) and \Cref{prop:AB00-II-DGS24-7.9}.(2). Hence the last part can be obtained from (3) and \Cref{prop:properties-cx-curv}.(2). 
\end{proof}


The following is a counterpart of \Cref{prop:tcx-tcurv-CI-ring} for Ext complexity and curvature.

\begin{proposition}\label{prop:injcx-injcurv-CI-ring}
    Let $M$ and $N$ be $R$-modules such that $M$ is Burch. Let $\depth(M)\ge 1$, and $\mathfrak{m}^h\Ext^{\gg0}_R(M,N)=0$ for some $h\ge 1$. Then 
     \begin{enumerate}[\rm (1)]
        \item 
        $\injcx_R(N) \le \cx_R(M,N)$ and $\injcurv_R(N) \le \curv_R(M,N)$. When $R$ is a complete intersection, both inequalities become equalities.
        \item 
        Assume $N$ is also a Burch module. Then the following conditions are equivalent.
        \begin{enumerate}[\rm (i)]
        \item $R$ is complete intersection $($resp., of codimension $c$$)$.
        \item $\cx_R(M,N) < \infty$ $($resp., $\cx_R(M,N) = c < \infty$$)$.
        \item $\curv_R(M,N) \le 1$.
        \end{enumerate} 
    \end{enumerate}
\end{proposition}

\begin{proof}
    (1) By \Cref{lem:DK23-3.6}, there exists an $M$-regular element $x$ of $R$ such that $k$ is a direct summand of $M/xM$. 
    Therefore  $\injcx_R(N)=\cx_R(k,N)\le \cx_R(M/xM,N)\le \cx_R(M,N)$, where the last two inequalities are obtained from \Cref{lem:cx-curv-direct-sum}.(3) and \cite[Lem.~3.13.(2)]{DGS24} respectively. Similarly, one can obtain $\injcurv_R(N) \le \curv_R(M,N)$. 
    
    When $R$ is complete intersection, due to \Cref{prop:AB00-II-DGS24-7.9}.(1), $\cx_R(M,N)\le \min\{\cx_R(M), \cx_R(N)\} \le \cx_R(N)=\injcx _R(N)$. Hence $\injcx _R(N) = \cx _R(M,N)$. For equality on the curvatures, since $0\le \injcurv_R(N) \le \curv_R(M,N)$, there is nothing to prove when $\curv_R(M,N)=0$. So we may assume that $\curv_R(M,N) > 0$, equivalently, $\cx_R(M,N)>0$. Therefore $\cx_R(M,N) = \injcx_R(N) < \infty$ by \Cref{prop:properties-cx-curv}(1).(i) and \ref{prop:properties-cx-curv}.(2). Hence, due to \Cref{rmk:cx-curv}.(4), $\curv_R(M,N) \le 1$. Thus $\curv_R(M,N)=1$. It remains to prove that $\injcurv_R(N) = 1$. Since $R$ is complete intersection, $\injcurv_R(N) =\curv_R(N) \le 1$, i.e., $\injcurv_R(N)$ can be $0$ or $1$. Since $\injcx_R(N)=\cx_R(M,N)>0$, it follows that $\injcurv_R(N)=1$.

    (2) When $R$ is complete intersection, in view of \Cref{prop:properties-cx-curv}.(2),
    \begin{equation}\label{cdim-equal-(inj)cx}
             \codim(R) = \cx_R(k) = \injcx_R(N) = \cx_R(M,N),
    \end{equation}
    where the last two equalities hold due to \Cref{prop:GS24-5.1}.(1) and
    (1) respectively. So, the implications (i) $\Rightarrow$ (ii) and (i) $\Rightarrow$ (iii) follow from \Cref{prop:AB00-II-DGS24-7.9}.(1), whereas (ii) $\Rightarrow$ (iii) is trivial (\Cref{rmk:cx-curv}.(4).(a)). Assuming (iii), by \Cref{prop:GS24-5.1}.(2) and (1) $\curv_R(k)=\curv_R(N) \le \curv_R(M,N) \le 1$. Hence $R$ is complete intersection by \Cref{prop:properties-cx-curv}.(2). Thus, one also sees that if $\cx_R(M,N) = c < \infty$, then $R$ is complete intersection of codimension $c$, see \eqref{cdim-equal-(inj)cx}.
\end{proof}

The following, which can be thought of as an Ext-counterpart to \Cref{thm:tcx-tcurv-CI-ring}, shows that when $M$ is a Burch submodule of $L$ of depth zero, in addition to $\mathfrak{m}^h\Ext^{\gg 0}_R(M,N)=0$, if $\mathfrak{m}^h\Ext^{\gg 0}_R(L,N)=0$ for some $h \ge 1$, then the following inequalities hold.

\begin{theorem}\label{thm:injcx-injcurv-CI-ring}
    Let $L,M$ and $N$ be $R$-modules such that $M$ is a Burch submodule of $L$. Suppose that $\mathfrak{m}^h\Ext^{\gg0}_R(M,N)=\mathfrak{m}^h\Ext_R^{\gg 0}(L,N)=0$ for some $h\ge 1$. Then 
    \begin{enumerate}[\rm (1)]
        \item 
        $\injcx_R(N) \le \max\{\cx_R(M,N), \cx_R(L/M,N) \}$.
        \item 
        $\injcurv_R(N) \le \max\{\curv_R(M,N), \curv_R(L/M,N) \}$.
        \item
        Both inequalities {\rm (1)} and {\rm (2)} become equalities if $R$ is a complete intersection.
        \item 
        Assume $N$ is also a Burch module. Then the following conditions are equivalent.
        \begin{enumerate}[\rm (i)]
        \item $R$ is complete intersection. 
        \item $\max\{\cx_R(M,N), \cx_R(L/M,N) \} < \infty$.
        \item $\max\{\curv_R(M,N), \curv_R(L/M,N) \} \le 1$.
    \end{enumerate}
    Under these conditions, $\codim(R) = \max\{\cx_R(M,N), \cx_R(L/M,N) \}$.
    \end{enumerate} 
\end{theorem}

\begin{proof}
    (1) When $\depth(M)\ge 1$, the inequality follows from \Cref{prop:injcx-injcurv-CI-ring}.(1). So assume that $\depth(M)=0$. Set $S$, $L'$ and $M'$ as defined in the proof of \Cref{thm:tcx-tcurv-CI-ring}.(1), and consider $N$ as an $S$-module via the isomorphism $S/XS\cong R$. Since $\m^h\Ext_R^{\gg 0}(M,N)=0$ and $\m^h\Ext_R^{\gg 0}(L,N)=0$, we obtain $\m^{2h}\Ext_R^{\gg 0}(L/M,N)=0$ by \Cref{lem:cx-curv-s.e.s}.(2). Hence, in view of \Cref{lem:Tor-Ext-iso-over-S-R}.(2), $\fn^{2h}\Ext_S^{\gg 0}(M',N)=0$, where $\fn$ is the maximal ideal of $S$. Note that $M'$ is a Burch $S$-submodule of $L'$ by \Cref{lem:GS24-lem-4.8}.(1). Therefore, since $\depth_S(M')\ge 1$, by \Cref{prop:injcx-injcurv-CI-ring}.(1), $\injcx_S(N) \le \cx_S(M',N)$. Due to \Cref{lem:cx-curv-N}.(2), $\injcx_S(N) = \injcx_R(N)$. Hence, by \Cref{lem:cx-curv-over-S-R}.(3), $\injcx_R(N) \le \cx_S(M',N) = \max\{ \cx_R(M,N), \cx_R(L/M,N) \}$.

    (2) Replacing $\cx$ by $\curv$, the proof follows along the same lines as (1). 

    (3) Assume that $R$ is complete intersection. Then, by \Cref{prop:AB00-II-DGS24-7.9}(1), we obtain $\cx_R(M,N) \le \min\{ \cx_R(M), \cx_R(N) \} \le \cx_R(N)=\injcx_R(N)$. Similarly, $\cx_R(L/M,N) \le \injcx_R(N)$. Thus $\max\{ \cx_R(M,N), \cx_R(L/M,N) \}\le \injcx_R(N)$. Hence, by (1), the equality follows. Analogous argument shows the equality for curvature.

    (4) In view of \Cref{prop:GS24-5.1}, we have $\injcx_R(N)=\cx_R(k)$ and $\injcurv_R(N) = \curv_R(k)$. So, the desired equivalences follow from (1), (2), \Cref{prop:properties-cx-curv}.(2) and \Cref{prop:AB00-II-DGS24-7.9}.(1). Hence, the last part can be obtained from (3) and \Cref{prop:properties-cx-curv}.(2). 
   %
\end{proof}

The following result is an analogue of \Cref{prop:injcx-injcurv-CI-ring}, where $N$ is assumed to be Burch in place of $M$.

\begin{proposition}\label{prop:cx-curv-CI-ring}
    Let $M$ and $N$ be $R$-modules such that $N$ is Burch. Let $\depth(N)\ge 1$, and $\mathfrak{m}^h\Ext_R^{\gg 0}(M,N)=0$ for some $h\ge 1$. Then 
    \begin{enumerate}[\rm (1)]
        \item 
        $\cx_R(M) \le \cx_R(M,N)$ and $\curv_R(M) \le \curv_R(M,N)$. When $R$ is complete intersection, both the inequalities become equalities.
        \item 
        Suppose $M$ is also a Burch module. Then the following are equivalent.
        \begin{enumerate}[\rm (i)]
        \item $R$ is complete intersection $($resp., of codimension $c$$)$.
        \item $\cx_R(M,N) < \infty$ $($resp., $\cx_R(M,N) = c < \infty$$)$.
        \item $\curv_R(M,N) \le 1$.
        \end{enumerate}
    \end{enumerate}
\end{proposition}
\begin{proof}
    (1) By \Cref{lem:DK23-3.6}, there exists an $N$-regular element $x$ of $R$ such that $k$ is a direct summand of $N/xN$. Therefore $\cx_R(M)=\cx_R(M,k)\le \cx_R(M,N/xN)\le \cx_R(M,N)$, where the last two inequalities are obtained by using \Cref{lem:cx-curv-direct-sum}.(5) and \cite[Lem.~3.13(1)]{DGS24} respectively. The inequality for curvature can be obtained using similar arguments. 

    When $R$ is complete intersection, due to \Cref{prop:AB00-II-DGS24-7.9}.(1), $\cx_R(M,N)\le \min\{\cx_R(M), \cx_R(N)\} \le \cx_R(M)$. Hence $\cx _R(M) = \cx _R(M,N)$. For equality on the curvatures, since $0\le \curv_R(M) \le \curv_R(M,N)$, there is nothing to prove when $\curv_R(M,N)=0$. So we may assume that $\curv_R(M,N) > 0$, equivalently, $\cx_R(M,N)>0$. Therefore $\cx_R(M,N) = \cx_R(M) < \infty$ by \Cref{prop:properties-cx-curv}(1).(i) and \ref{prop:properties-cx-curv}.(2). Hence, due to \Cref{rmk:cx-curv}.(4), $\curv_R(M,N) \le 1$. Thus $\curv_R(M,N)=1$. It remains to prove that $\curv_R(M) = 1$. Since $R$ is complete intersection, $\curv_R(M) \le 1$, i.e., $\curv_R(M)$ can be $0$ or $1$. Since $\cx_R(M)=\cx_R(M,N)>0$, it follows that $\curv_R(M)=1$. 

    (2) When $R$ is complete intersection, in view of \Cref{prop:properties-cx-curv}.(2),
    \begin{equation}\label{cdim-equal-cx}
             \codim(R) = \cx_R(k) = \cx_R(M) = \cx_R(M,N),
    \end{equation}
    where the last two equalities hold due to \Cref{prop:GS24-5.1}.(1) and
    (1) respectively. So, the implications (i) $\Rightarrow$ (ii) and (i) $\Rightarrow$ (iii) follow from \Cref{prop:AB00-II-DGS24-7.9}.(1), whereas (ii) $\Rightarrow$ (iii) is trivial (\Cref{rmk:cx-curv}.(4).(a)). Now, assuming (iii), by \Cref{prop:GS24-5.1}.(2) and (1) $\curv_R(k)=\curv_R(M) \le \curv_R(M,N) \le 1$. Hence $R$ is complete intersection by \Cref{prop:properties-cx-curv}.(2). Thus, one also sees that if $\cx_R(M,N) = c < \infty$, then $R$ is complete intersection of codimension $c$, see \eqref{cdim-equal-cx}.
    %
\end{proof}

Under the assumptions that $N$ is a Burch submodule of $L$ and $\mathfrak{m}^h\Ext_R^{\gg 0}(M,N)= \m^h\Ext_R^{\gg 0}(M,L)=0$ for some $h\ge 1$, the following result is analogous to \Cref{thm:injcx-injcurv-CI-ring}.

\begin{theorem}\label{thm:cx-curv-CI-ring}
    Let $L,M$ and $N$ be $R$-modules such that $N$ is a Burch submodule of $L$. Let $\mathfrak{m}^h\Ext_R^{\gg 0}(M,N)= \m^h\Ext_R^{\gg 0}(M,L)=0$ for some $h\ge 1$. Then 
        \begin{enumerate}[\rm (1)]
        \item 
        $\cx_R(M) \le \max\{\cx_R(M,N), \cx_R(M,L/N) \}$.
        \item 
        $\curv_R(M) \le \max\{\curv_R(M,N), \curv_R(M,L/N) \}$.
        \item
        Both inequalities {\rm (1)} and {\rm (2)} become equalities if $R$ is a complete intersection.
        \item 
        Suppose $M$ is also a Burch module. Then the following are equivalent.
        \begin{enumerate}[\rm (i)]
        \item $R$ is complete intersection.
        \item $\max\{\cx_R(M,N), \cx_R(M,L/N) \} < \infty$.
        \item $\max\{\curv_R(M,N), \curv_R(M,L/N) \} \le 1$.
    \end{enumerate}
    Under these conditions, $\codim(R) = \max\{\cx_R(M,N), \cx_R(M,L/N) \}$.
    \end{enumerate}
\end{theorem}
\begin{proof}
    (1) When $\depth(N)\ge 1$, the inequality follows from \Cref{prop:cx-curv-CI-ring}.(1). Therefore, assume that $\depth(N)=0$. Consider $S:=R[X]_{\langle \m,X\rangle}$, $L':= L[X]_{\langle \m,X\rangle}$ and $N':=(N+XL[X])_{\langle \m,X\rangle}$. We may view $M$ as $S$-module via the isomorphism $S/XS\cong R$. Since $\m^h\Ext_R^{\gg 0}(M,N)=0$ and $\m^h\Ext_R^{\gg 0}(M,L)=0$, it follows from \Cref{lem:cx-curv-s.e.s}.(3) that $\m^{2h}\Ext_R^{\gg 0}(M,L/N)=0$. Therefore, by \Cref{lem:Tor-Ext-iso-over-S-R}(3), we get $\fn^{2h}\Ext_R^{\gg 0}(M,N')=0$, where $\fn$ is the maximal ideal of $S$. Note that $N'$ is a Burch $S$-submodule of $L'$ by \Cref{lem:GS24-lem-4.8}.(1). Therefore, since $\depth_S(N')\ge 1$, by \Cref{prop:cx-curv-CI-ring}.(1), $\cx_S(M) \le \cx_S(M,N')$. Due to \Cref{lem:cx-curv-N}.(1), $\cx_S(M) = \cx_R(M)$. Hence, by \Cref{lem:cx-curv-over-S-R}.(5), $\cx_R(M) \le \cx_S(M,N') = \max\{ \cx_R(M,N), \cx_R(M,L/N) \}$. 

    (2) Replacing $\cx$ by $\curv$, the proof follows along the same lines as (1).

    (3) Now, let $R$ be complete intersection. From \Cref{prop:AB00-II-DGS24-7.9}(1), $\cx_R(M,N) \le \min\{ \cx_R(M), \cx_R(N) \} \le \cx_R(M)$. Similarly, $\cx_R(M,L/N) \le \cx_R(M)$. Thus
    \[
        \max\{\cx_R(M,N), \cx_R(M,L/N)\} \le \cx_R(M).
    \]
    Hence, by (1), the equality follows. The analogous argument shows the equality for curvature.

    (4) In view of \Cref{prop:GS24-5.1}, we have $\cx_R(M)=\cx_R(k)$ and $\curv_R(M) = \curv_R(k)$. So the desired equivalences follow from (1), (2), \Cref{prop:properties-cx-curv}.(2) and \Cref{prop:AB00-II-DGS24-7.9}.(1). Hence the last part can be obtained using (3) and \Cref{prop:properties-cx-curv}.(2).
    %
    %
\end{proof}

The next three propositions show that, under certain conditions, when one of the modules in the pair $(M,N)$ is Burch, the (injective) complexity (resp., curvature) of the other module is bounded above by the  Ext or Tor complexity (resp., curvature) of the pair.

\begin{proposition}\label{prop:tcx-le}
    Let $L,M$ and $N$ be $R$-modules such that $M$ is a Burch submodule of $L$. Suppose that $\m^h\Tor_{\gg 0}^R(L,N)=0$ and $\m^h\Tor_{\gg 0}^R(M,N)=0$ for some $h>0$.
    \begin{enumerate}[\rm (1)]
        \item 
        If $\tcx_R(L,N)\le \tcx_R(M,N)$, then $\cx_R(N)\le \tcx_R(M,N)$.
        \item 
        If $\tcurv_R(L,N)\le \tcurv_R(M,N)$, then $\curv_R(N)\le \tcurv_R(M,N)$.
    \end{enumerate}
\end{proposition}

\begin{proof}
    In view of \Cref{lem:cx-curv-s.e.s}.(1), the short exact sequence $0\to M\to L\to L/M\to 0$ yields $\tcx_R(L/M,N)\le \sup\{\tcx_R(M,N),\tcx_R(L,N)\}=\tcx_R(M,N)$ and $\tcurv_R(L/M,N)\le \sup\{\tcurv_R(M,N),\tcurv_R(L,N)\}=\tcurv_R(M,N)$. Hence, we are done by \Cref{thm:tcx-tcurv-CI-ring}.(1) and \Cref{thm:tcx-tcurv-CI-ring}.(2).
\end{proof}

\begin{proposition}\label{prop:cx-le}
    Let $L,M$ and $N$ be $R$-modules such that $M$ is a Burch submodule of $L$. Suppose that $\m^h\Ext_R^{\gg 0}(L,N)=0$ and $\m^h\Ext_R^{\gg 0}(M,N)=0$ for some $h>0$.
    \begin{enumerate}[\rm (1)]
        \item 
        If $\cx_R(L,N)\le \cx_R(M,N)$, then $\injcx_R(N)\le \cx_R(M,N)$.
        \item 
        If $\curv_R(L,N)\le \curv_R(M,N)$, then $\injcurv_R(N)\le \curv_R(M,N)$.
    \end{enumerate}
\end{proposition}

\begin{proof}
     The proof is similar to \Cref{prop:tcx-le} using \Cref{lem:cx-curv-s.e.s}.(2) and \Cref{thm:injcx-injcurv-CI-ring} in places of \Cref{lem:cx-curv-s.e.s}.(1) and \Cref{thm:tcx-tcurv-CI-ring} respectively.
\end{proof}

\begin{proposition}\label{prop:cx-le-2}
    Let $L,M$ and $N$ be $R$-modules such that $N$ is a Burch submodule of $L$. Suppose that $\m^h\Ext_R^{\gg 0}(M,L)=0$ and $\m^h\Ext_R^{\gg 0}(M,N)=0$ for some $h>0$.
    \begin{enumerate}[\rm (1)]
        \item 
        If $\cx_R(M,L)\le \cx_R(M,N)$, then $\cx_R(M)\le \cx_R(M,N)$.
        \item 
        If $\curv_R(M,L)\le \curv_R(M,N)$, then $\curv_R(M)\le \curv_R(M,N)$.
    \end{enumerate}
\end{proposition}

\begin{proof}
    In view of \Cref{lem:cx-curv-s.e.s}.(3), the short exact sequence $0\to N\to L\to L/N\to 0$ yields $\cx_R(M,L/N)\le \sup\{\cx_R(M,N),\cx_R(M,L)\}=\cx_R(M,N)$ and $\curv_R(M,L/N)\le \sup\{\curv_R(M,N),\curv_R(M,L)\}=\curv_R(M,N)$. Hence, we are done by \Cref{thm:cx-curv-CI-ring}.(1) and \Cref{thm:cx-curv-CI-ring}.(2).
\end{proof}

The next result shows that if the base ring is CM, 
$N$ is Burch, and some power of the maximal ideal annihilates $\Ext_R^{\gg 0}(M,N)$, then the complexity (resp., curvature) of $M$ is at most that of the pair $(M,N)$.

\begin{theorem}\label{thm:R-CM}
    Assume that $R$ is CM. Let $M$ and $N$ be $R$-modules such that $N$ is Burch, and $\m^h \Ext_R^{\gg 0}(M,N)=0$ for some $h>0$. Then $\cx_R(M)\le \cx_R(M,N)$ and $\curv_R(M)\le \curv_R(M,N)$.   
\end{theorem}

\begin{proof}
    All hypotheses and conclusions are preserved under passing to completion (for the Burch property, see \Cref{burchcomplete}). Hence, we may assume that $R$ is complete. Let $N$ be a Burch submodule of an $R$-module $X$. Since $R$ is complete and CM, there exist a finitely generated $R$-module $L$ of finite injective dimension and an injective homomorphism $h: X \to L$ (see \cite[11.8 and 11.17]{lw}). By \Cref{lem:Burch-comp}, $h(N)$ is a Burch submodule of $L$. Note that $\Ext^{\gg 0}_R(M,L)=0$ (i.e., $\cx_R(M,L)=0$), and $h(N)\cong N$. Therefore, by \Cref{prop:cx-le-2}, $\cx_R(M)\le \cx_R(M,h(N))=\cx_R(M,N)$, and similarly $\curv_R(M)\le \curv_R(M,N)$.
\end{proof}

Before we proceed to the next result, we require the following preparatory lemma.

\begin{lemma}\label{lem:MCM-appr-Ext-iso}
    Suppose that $R$ is CM of dimension $d$, and it admits a canonical module. Let $M$ and $N$ be $R$-modules such that $N$ locally has finite injective dimension on the punctured spectrum. Then, there exist MCM $R$-modules $X$ and $Y$ such that $Y$ locally has finite injective dimension on punctured spectrum, and $\Ext^{n+d}_R(M,N)\cong \Ext_R^n(X,Y)$ for all $n\ge 1$.
\end{lemma}

\begin{proof}
    Set $X:=\syz^d_R(M)$. Then $X$ is MCM, and $\Ext^{n+d}_R(M,N) \cong \Ext_R^n(X,N)$ for all $n\ge 1$. Consider an MCM approximation (cf.~\cite[11.8 and 11.17]{lw}) of $N$, i.e., a short exact sequence $0\to L \to Y \to N\to 0$, where $L$ has finite injective dimension, and $Y$ is MCM. Hence, the desired isomorphisms follow because $\Ext_R^n(X,L)=0$ for all $n\ge 1$. Note that $Y$ locally has finite injective dimension on punctured spectrum as both $L$ and $N$ have the same property.
\end{proof}

When the completion $\widehat N$ locally has finite injective dimension on the punctured spectrum of $\widehat R$, we establish a result similar to \Cref{thm:R-CM}.

\begin{proposition}\label{gornew}
    Suppose that $R$ is CM. Let $M$ and $N$ be $R$-modules such that $N$ is Burch. If $\widehat N$ locally has finite injective dimension on the punctured spectrum of $\widehat R$, then $\cx_R(M)\le \cx_R(M,N)$ and $\curv_R(M)\le \curv_R(M,N)$.
\end{proposition}

\begin{proof}
    Set $d:=\dim(R)$. As in the proof of \Cref{thm:R-CM}, passing to completion, we may assume that $R$ is complete, and $N$ locally has finite injective dimension on the punctured spectrum. Hence, since $R$ is CM, it admits a canonical module, say $\omega$. Set $(-)^{\dagger}:= \Hom_R(-,\omega)$. By \Cref{lem:MCM-appr-Ext-iso}, there exist MCM $R$-modules $X$ and $Y$ such that $Y$ locally has finite injective dimension on the punctured spectrum, and $\Ext^{n+d}_R(M,N)\cong \Ext_R^n(X,Y)\cong \Ext_R^n(Y^{\dagger}, X^{\dagger})$ for all $n\ge 1$, where the last isomorphism is by duality of dualizing module, see, e.g., \cite[3.14 and 3.15]{cluster}. Since $Y$ is MCM and locally has finite injective dimension on the punctured spectrum, for every non-maximal prime ideal $\mathfrak p$ of $R$, the localization $Y_{\mathfrak p}$ is a finite direct sum of copies of $\omega_{\mathfrak p}$. Thus, $Y^\dagger$ is locally free on the punctured spectrum of $R$. Therefore, by \cite[Lem.~3.6]{DGS24}, there exists $h>0$ such that $0=\m^h\Ext_R^n(Y^{\dagger}, X^{\dagger})=\m^h\Ext_R^{n+d}(M,N)$ for all $n\gg 0$. Hence, the proposition follows from \Cref{thm:R-CM}. 
\end{proof} 

As a consequence of the above proposition, we obtain the following. 

\begin{theorem}\label{cornew}
    Let $R$ be CM such that $\widehat R$ is locally Gorenstein on the punctured spectrum. Let $M$ and $N$ be $R$-modules such that $N$ is Burch. Assume that $N$ locally has finite projective dimension on the punctured spectrum. Then $\cx_R(M)\le \cx_R(M,N)$ and $\curv_R(M)\le \curv_R(M,N)$. 
\end{theorem}

\begin{proof}
    In view of \Cref{newloc}, $\widehat N$ locally has finite projective dimension on the punctured spectrum of $\widehat R$. Hence, since $\widehat R$ is locally Gorenstein on the punctured spectrum, $\widehat N$ locally has finite injective dimension on the punctured spectrum of $\widehat R$. Therefore the proof follows from \Cref{gornew}. 
\end{proof}

In general, it may happen that $R$ is CM with isolated singularity, but $\widehat R$ is locally not even Gorenstein on the punctured spectrum, see \cite[Main Theorem]{heit}. In the following remark, we note a class of local rings whose completions are locally Gorenstein on the punctured spectrum.

\begin{remark}\label{rmk:loc-gor-punc}
\begin{enumerate}[\rm (1)]
    \item 
    If $R$ is Gorenstein, then $\widehat R$ is Gorenstein, hence $\widehat R$ is locally Gorenstein on the punctured spectrum. If $R$ is Artinian, then trivially $\widehat R$ is locally Gorenstein on the punctured spectrum. 
    \item 
    Suppose that $R$ is CM of dimension $1$. Then it is well known that $\widehat R$ is locally Gorenstein on the punctured spectrum if and only if $R$ admits a canonical ideal. Indeed, in view of \cite[Prop.~2.7]{agl}, $R$ admits a canonical ideal \iff the total ring of fractions $Q(\widehat{R})$ is Gorenstein \iff $\widehat R$ is locally Gorenstein at all minimal primes of $\widehat R$, where the last equivalence follows from \cite[Lem.~10.25.4]{stacks-project} and the fact that the direct product of finitely many rings is Gorenstein \iff all the components are Gorenstein.
    \item 
    If $R$ is CM that admits a canonical module, then $\widehat R$ is locally Gorenstein on the punctured spectrum if and only if $R$ is locally Gorenstein on the punctured spectrum; this can be obtained by applying \Cref{newloc} to the canonical module.
\end{enumerate}
\end{remark}



When $M$ is Burch submodule of an $R$-module $X$ which either has finite $G$-dimension or is torsionless, the following is an analogue of \Cref{thm:R-CM}.

\begin{theorem}\label{new}
    Let $M$ and $N$ be $R$-modules such that $M$ is a Burch submodule of some $R$-module $X$. Assume that $\gdim_R(X)<\infty$, or $X$ is torsionless. Then, the following hold.
    \begin{enumerate}[\rm(1)]
        \item 
        If $\m^h\Tor^R_{\gg 0}(M,N)=0$ for some $h>0$, then $\cx_R(N)\le \tcx_R(M,N)$ and $\curv_R(N)\le \tcurv_R(M,N)$.
        \item 
        If $\m^h\Ext_R^{\gg 0}(M,N)=0$ for some $h>0$, then $\injcx_R(N)\le \cx_R(M,N)$ and $\injcurv_R(N)\le \curv_R(M,N)$.
\end{enumerate}
\end{theorem}

\begin{proof}
    First, assume that $\gdim_R(X)<\infty$. Then, in view of \Cref{para:Aus-Buch-Thm-1.1}, there exists an $R$-module $L$ of finite projective dimension and an injective homomorphism $h : X \to L$. By \Cref{lem:Burch-comp}, $h(M)$ is a Burch submodule of $L$. Note that $\Tor^R_{\gg 0}(L,N)=0=\Ext^{\gg 0}_R(L,N)$, and $h(M)\cong M$. Therefore, we are done by \Cref{prop:tcx-le,prop:cx-le}.

    For the second case, assume that $X$ is torsionless. Then, there exist a free $R$-module $L$ and an injective homomorphism $h: X \to L$. Hence, the proof is similar as above.
\end{proof}

\begin{remark}\label{rmk:van-cond}
\begin{enumerate}[(1)]
    \item 
    The hypotheses $\m^h\Tor^R_{\gg 0}(M,N)=0$ and/or $\m^h\Tor^R_{\gg 0}(L,N)=0$ for some $h\ge 1$ in Propositions~\ref{prop:tcx-tcurv-CI-ring}, \ref{prop:tcx-le} and Theorems~\ref{thm:tcx-tcurv-CI-ring}, \ref{new}.(1) hold if $N$ locally has finite projective dimension on the punctured spectrum; see \cite[Lem.~3.6.(2)]{DGS24}.
    \item 
    The hypotheses $\mathfrak{m}^h\Ext_R^{\gg 0}(M,N)=0$ and/or $\m^h\Ext_R^{\gg 0}(M,L)=0$ for some $h\ge 1$ in Propositions~\ref{prop:injcx-injcurv-CI-ring}, \ref{prop:cx-curv-CI-ring}, \ref{prop:cx-le-2} and Theorems~\ref{thm:cx-curv-CI-ring}, \ref{thm:R-CM}, \ref{new}.(2) hold if $M$ locally has finite projective dimension on the punctured spectrum; see \cite[Lem.~3.6.(1)]{DGS24}.
    \item 
    Thus, when $R$ has an isolated singularity, the annihilation conditions on Tor/Ext modules in Propositions~\ref{prop:tcx-tcurv-CI-ring}, \ref{prop:injcx-injcurv-CI-ring}, \ref{prop:cx-curv-CI-ring}, \ref{prop:tcx-le}, \ref{prop:cx-le}, \ref{prop:cx-le-2} and Theorems~\ref{thm:tcx-tcurv-CI-ring}, \ref{thm:injcx-injcurv-CI-ring}, \ref{thm:cx-curv-CI-ring}, \ref{thm:R-CM}, \ref{new} are immediate.
\end{enumerate}
\end{remark}




 
The following proposition illustrates the usefulness of \Cref{lem:Burch-comp} by improving the result \cite[Cor.~5.4]{GS24}.

\begin{proposition}\label{vanish}
    Let $M$ and $N$ be $R$-modules such that $M$ is a Burch submodule of some $R$-module $X$. Suppose $n\ge \max\{1,\depth(N)\}$ is an integer such that $\Ext^i_R(M,N)=0$ for $i=n-1,n,n+1$, and at least one of the following holds.
    \begin{enumerate}[\rm(1)]
     \item $\gdim_R(X)<\infty$, and $n>\depth(R)$.
     \item $X$ is torsionless.
    \end{enumerate}
    Then, $\id_R(N)<\infty$.
\end{proposition}

\begin{proof}
    (1) In view of \Cref{para:Aus-Buch-Thm-1.1}, there exist an $R$-module $L$ of finite projective dimension and an injective homomorphism $h: X \to L$. By \Cref{lem:Burch-comp}, $h(M)$ is a Burch submodule of $L$. Since $n>\depth(R)\ge \pd_R(L)$, it follows that $\Ext_R^{i}(L,N)=0$ for $i=n, n+1$. Since $h(M)\cong M$, by \cite[Thm.~5.3]{GS24}, $\id_R(N)<\infty$. 

    (2) Since $X$ is torsionless, there exist a free $R$-module $L$  and an injective homomorphism $h: X \to L$. Note that $\Ext_R^{\ge 1}(L,N)=0$. Hence, as in (1), $\id_R(N)<\infty$.
\end{proof}

We now provide the proofs of Corollaries~\ref{cor:char-CI-via-burch-ideals} and \ref{cor:char-CI-via-burch-ideals-in-Gor-ring}.

\begin{proof}[Proof of \Cref{cor:char-CI-via-burch-ideals}]
    Note that every $\m$-primary ideal is locally free on the punctured spectrum. Since $R$ is torsionless, $I$ is $\m$-primary and Burch, by \Cref{thm2}, $\cx_R(J)\le \tcx_R(I,J)$ and $\injcx_R(J)\le \cx_R(I,J)$. Therefore, since $J$ is also Burch, in view of Proposition~\ref{prop:GS24-5.1}.(1),
    \begin{equation}\label{eqn-cx-k-pair}
        \cx_R(k)\le \tcx_R(I,J) \; \mbox{ and } \; \cx_R(k)\le \cx_R(I,J).
    \end{equation}
    Thus, if $\cx_R(k)=\infty$, then clearly $\cx_R(k)=\tcx_R(I,J)=\cx_R(I,J)$. So we may assume that $\cx_R(k)$ is finite. Then, by Proposition~\ref{prop:properties-cx-curv}.(2), $R$ is complete intersection. Hence, by Propositions~\ref{prop:AB00-II-DGS24-7.9} and \ref{prop:properties-cx-curv}.(1).(i), $\tcx_R(I,J)\le \cx_R(k)$ and $\cx_R(I,J)\le \cx_R(k)$. Thus $\tcx_R(I,J)=\cx_R(I,J)=\cx_R(k)$. In view of Proposition~\ref{prop:properties-cx-curv}.(2), the equivalence of (1), (2) and (3) follows from these equalities. Arguing similarly as in \eqref{eqn-cx-k-pair}, \Cref{thm2} and Proposition~\ref{prop:GS24-5.1}.(2) yield $\curv_R(k)\le \tcurv_R(I,J)$ and $\curv_R(k)\le \curv_R(I,J)$. Hence, the implications $(4)\implies (1)$ and $(5)\implies (1)$ follow from Proposition~\ref{prop:properties-cx-curv}.(2). The reverse implications are obvious from \Cref{prop:AB00-II-DGS24-7.9}.
\end{proof}

\begin{proof}[Proof of \Cref{cor:char-CI-via-burch-ideals-in-Gor-ring}] 
    In view of \Cref{prop:GS24-5.1} and \Cref{cornew}, $\cx_R(k)=\cx_R(I) \le \cx_R(I,J)$ and $\curv_R(k)=\curv_R(I)\le \curv_R(I,J)$. Thus, if $\cx_R(k)=\infty$, then $\cx_R(k)=\cx_R(I,J)$. Therefore, we may assume that $\cx_R(k)$ is finite. By \Cref{prop:properties-cx-curv}.(2), this implies that $R$ is a complete intersection. Hence, Propositions~\ref{prop:AB00-II-DGS24-7.9}.(1) and \ref{prop:properties-cx-curv}.(1) yield $\cx_R(k) \le \cx_R(I,J) \le \min\{\cx_R(I), \cx_R(J)\} \le \cx_R(k)$. Consequently, $\cx_R(I,J) = \cx_R(k)$. Thus, the equivalence of (1) and (2) follows from Proposition~\ref{prop:properties-cx-curv}.(2). Since $\curv_R(k)\le \curv_R(I,J)$, the equivalence of (1) and (3) can be observed from Propositions~\ref{prop:properties-cx-curv}.(2) and \ref{prop:AB00-II-DGS24-7.9}.(1).
\end{proof}

\section{Some examples}\label{sec:exam}

Here we provide two examples computing complexity and curvature for pairs of Burch ideals to complement our results.

\begin{example}\label{exam:m-power-2=0}
    Let $(R,\m,k)$ be a Noetherian local ring such that $\m\neq 0$ and $\m^2=0$. Set $b:=\mu(\m)$. For example, $R$ can be $k[x_1,\ldots,x_b]/(x_ix_j : 1\le i\le j\le b)$, where $x_1,\ldots,x_b$ are indeterminates over a field $k$. Note that $\m$ is a Burch ideal of $R$. Moreover, 
    \begin{equation*}
        \tcx_R(\m,\m) = \cx_R(\m,\m) = \cx_R(k) =
        \begin{cases}
            1 & \mbox{if $b = 1$,}\\
            \infty & \mbox{if  $b \ge 2$,}
        \end{cases}
    \end{equation*}
    and $\tcurv_R(\m,\m) = \curv_R(\m,\m) = \curv_R(k) = b$. Clearly, $R$ is complete intersection \iff $b=1$ \iff $R$ is a hypersurface.
\end{example}

\begin{proof}
    Note that $\m\cong k^{\oplus b}$. Then $\beta_n^R(k)=\beta_{n-1}^R(\m)=b\beta_{n-1}^R(k)$ for all $n\ge 1$. Thus $\beta_n^R(k)=b\beta_{n-1}^R(k) =b^2\beta_{n-2}^R(k) = \cdots = b^{n}\beta_0^R(k) = b^{n}$ for all $n$. Now $\Tor_n^R(\m,\m)\cong \Tor_n^R(k^{\oplus b},k^{\oplus b})\cong \bigoplus_b\bigoplus_b\Tor_n^R(k,k)$. Similarly, $\Ext^n_R(\m,\m)\cong \bigoplus_b\bigoplus_b\Ext^n_R(k,k)$. So $\mu_R(\Tor_n^R(\m,\m))=\mu_R(\Ext^n_R(\m,\m))=b^2\beta_n^R(k)=b^{n+2}$ for all $n$. Hence, computing complexity and curvature, the desired equalities follow.
\end{proof}
 
    

Our next example shows that the $\m$-primary condition on $I$ is not a necessary condition in \Cref{cor:char-CI-via-burch-ideals}. This leads to \Cref{ques:omit-m-primary}.

\begin{example}\label{exam:omit-m-primary}
    Fix an integer $b\ge1$. Let $R:=k[x_1,\ldots,x_b]/(x_ix_j : 1\le i <j\le b)$, where $x_1,\ldots,x_b$ are indeterminates over a field $k$. Set $\m:=(x_1,\ldots,x_b)$. Let $\ell,m\ge 1$ be two integers. Then $\m^{\ell}$ and $(x_i^{\ell+1}) = x_i^{\ell}\m$ (for $1\le i \le b$) are Burch ideals, however $(x_i^{\ell})$ are not $\m$-primary.
    For $1 \le i, j \le b$, we have
    \begin{align*}
    & \tcx_R((x_i^{\ell}),(x_j^{m})) = \tcx_R((x_i^{\ell}),\m^m) =\tcx_R(\m^{\ell},\m^m) = \cx_R((x_i^{\ell}),(x_j^{m})) \\ & = \cx_R((x_i^{\ell}),\m^m) = \cx_R(\m^{\ell},(x_j^{m})) = \cx_R(\m^{\ell},\m^m) = \cx_R(k) = 
    \begin{cases}
        b-1 & \mbox{if $1\le b \le 2$,}\\ 
        \infty & \mbox{if  $b \ge 3$,}
    \end{cases}
    \end{align*}
    and
    \begin{align*}
        & \tcurv_R((x_i^{\ell}),(x_j^{m})) = \tcurv_R((x_i^{\ell}),\m^m) = \tcurv_R(\m^{\ell},\m^m) = \curv_R((x_i^{\ell}),(x_j^{m})) \\
        & = \curv_R((x_i^{\ell}),\m^m) = \curv_R(\m^{\ell},(x_j^{m})) = \curv_R(\m^{\ell},\m^m) = \curv_R(k) = b-1.
    \end{align*}
\end{example}

\begin{proof}
    Fix $1 \le i, j \le b$. Observe that $\ann_R(x_i^\ell)=\ann_R(x_i)$. Therefore, $(x_i^\ell) \cong R/ \ann_R(x_i^\ell) \cong R/ \ann_R(x_i) \cong (x_i)$. Consequently, $\m^{\ell} \cong (x_1^{\ell})\oplus \dots \oplus(x_b^{\ell}) \cong (x_1)\oplus \dots \oplus(x_b) \cong \m$. So we need to compute complexity and curvature involving $(x_i)$, $(x_j)$ and $\m$.
    
    Since $\m \cong (x_1)\oplus \dots \oplus(x_b)$, it follows that
    \begin{equation}\label{Ext-iso}
        \bigoplus_{r=1}^b\Ext^n_R((x_r),\m) \cong \Ext^n_R(\m,\m) \cong \bigoplus_{s=1}^b\Ext^n_R(\m, (x_s)) \cong \bigoplus_{r,s=1}^b\Ext^n_R((x_r),(x_s)).
    \end{equation}
    Hence, by \Cref{lem:DGS-2.7-2.2}, $\cx_R(\m,\m) = \sup\{\cx_R((x_r),\m) : 1\le r \le b\} = \cx_R((x_i),\m)$ by symmetry of the variables. Similarly, $\cx_R(\m,\m) = \cx_R(\m,(x_j))$. Moreover, the isomorphisms \eqref{Ext-iso} yield that
    $\cx_R(\m,\m) = \sup\{\cx_R((x_r),(x_s)) : 1\le r,s \le b\}$. By symmetry of the variables, note that $\cx_R((x_r),(x_s))$, for $1\le r,s \le b$ with $r\neq s$, are equal. Also $\cx_R((x_r),(x_r))$, for $1\le r \le b$, are equal. Since $\Omega_R^1(x_r) = \bigoplus_{1\le s\le b, s\neq r} (x_s)$, in view of \Cref{lem:DGS-2.7-2.2}, $\cx_R((x_r),(x_r)) = \cx_R(\Omega_R^1(x_r),(x_r)) = \sup\{\cx_R((x_s),(x_r)) : 1\le s\le b, s\neq r \} = \cx_R((x_s),(x_r))$ for every $s\neq r$. Thus, $\cx_R(\m,\m) = \cx_R((x_i),(x_j))$. Combining all these equalities, $\cx_R(\m,\m) = \cx_R((x_i),\m) = \cx_R(\m,(x_j)) = \cx_R((x_i),(x_j))$. Similarly, one obtains the corresponding equalities for curvature, Tor-complexity and Tor-curvature.

    Since $\m = \Omega_R^1(k)$ and $\m$ is a Burch ideal, in view of \Cref{prop:GS24-5.1}, $\cx_R(\m,\m)= \cx_R(k,\m)= \injcx_R(\m)=\cx_R(k)$ and $\tcx_R(\m,\m)= \tcx_R(k,\m)= \cx_R(\m)=\cx_R(k)$. Similarly, $\curv_R(\m,\m)=\curv_R(k)=\tcurv_R(\m,\m)$, which is same as $\curv_R(\m)$. So, in order to prove the desired equalities, it is enough to compute $\cx_R(\m)$ and $\curv_R(\m)$.

    Since $\Omega_R^1(x_r) = \bigoplus_{1\le s\le b, s\neq r} (x_s)$, it follows that $\beta_1^R((x_r)) =\mu(\Omega_R^1(x_r))=b-1$ for $1\le r \le b$. Hence
    \begin{center}
        $\beta_2^R(x_r)=\beta_1^R(\Omega_R^1(x_r))=\beta_1^R(\bigoplus_{1\le s\le b, s\neq r} (x_s))= \sum_{1\le s\le b, s\neq r} \beta_1^R((x_s))=(b-1)^2$.
    \end{center}
    Thus, using induction, for each $n\ge 1$, we obtain that $\beta_n^R(x_r)=\beta_{n-1}^R(\Omega^1_R(x_r))=(b-1)^n$ for $1\le r\le b$. Therefore $\beta_n^R(\m)=\beta_n^R\big((x_1)\oplus \dots \oplus(x_b)\big)=\sum_{r=1}^b\beta_n^R(x_r)=b(b-1)^n$ for all $n\ge 1$. So
    \[
    \cx_R(\m)=
    \begin{cases}
        0 & \mbox{if $b = 1$,}\\ 
        1 & \mbox{if $b = 2$,}\\
        \infty & \mbox{if  $b \ge 3$,}
    \end{cases}
    \]
    and $\curv_R(\m)=b-1$. Thus the required equalities follow.
    \mycomment{
    Note that for each $\alpha \in \{1, \ldots, b\}$ and $\ell \ge 1$, $(x_{\alpha}^{\ell})$ is integrally closed. Indeed, consider a non-zero element $r\in R$ which is integral over $(x_{\alpha}^{\ell})$. Then
    \begin{equation}\label{equn:integral dependence}
        r^n+a_1r^{n-1}+\dots +a_n=0, \; \mbox{where $a_i \in (x_{\alpha}^{\ell i})$, $1 \le i \le n$,}
    \end{equation}
    for some integer $n$. Clearly, $r^n \in (x_{\alpha})$, and hence $r\in (x_{\alpha})$. Since $x_{\alpha}x_j=0$ for $j\neq \alpha$, in view of \eqref{equn:integral dependence}, $r$ has the form $r=b_tx_{\alpha}^t+b_{t+1}x_{\alpha}^{t+1}+ \dots +b_mx_{\alpha}^m$ for some $m\ge t \ge 1$, where $b_i \in k$ and $b_t \neq 0$. The initial degree of $r^n$ is $nt$. Since $a_i \in (x_{\alpha}^{\ell i})$, the initial degree of $a_i$ is greater than or equal to $\ell i$ for $1 \le i \le n$. So the initial degree of $a_ir^{n-i}$ is greater than or equal to $\ell i + (n-i)t = nt+i(\ell -t)$. Therefore, comparing the initial degrees of $r^n$ and $a_ir^{n-i}$ (for $1 \le i \le n$), one obtain that $t \ge \ell$. Hence $r \in (x_{\alpha}^{\ell})$.}
\end{proof}


We close this article by presenting the following natural questions.




\begin{question}\label{ques:omit-m-primary}
    Does \Cref{cor:char-CI-via-burch-ideals} hold even if $I$ is not $\m$-primary?
\end{question}

\begin{question}\label{ques:omit-Gor}
    Does \Cref{cor:char-CI-via-burch-ideals-in-Gor-ring} hold even if $\widehat R$ is not locally Gorenstein on punctured spectrum?
\end{question}

\bibliographystyle{plain}
\bibliography{main}

\end{document}